\documentclass[11pt]{amsart}

\usepackage[a4paper,margin=1.15in]{geometry}
\usepackage[T1]{fontenc}
\usepackage[utf8]{inputenc}
\usepackage{lmodern}
\usepackage{amsmath,amssymb,amsthm,mathtools}
\usepackage{silence}
\usepackage{microtype}
\usepackage{enumitem}
\usepackage{aliascnt}
\usepackage{hyperref}
\hypersetup{
  colorlinks=true,
  linkcolor=blue,
  citecolor=blue,
  urlcolor=blue
}
\usepackage[nameinlink,noabbrev]{cleveref}

\numberwithin{equation}{section}

\newtheorem{theorem}{Theorem}[section]

\newaliascnt{proposition}{theorem}
\newtheorem{proposition}[proposition]{Proposition}
\aliascntresetthe{proposition}

\newaliascnt{lemma}{theorem}
\newtheorem{lemma}[lemma]{Lemma}
\aliascntresetthe{lemma}

\newaliascnt{corollary}{theorem}
\newtheorem{corollary}[corollary]{Corollary}
\aliascntresetthe{corollary}

\newaliascnt{definition}{theorem}
\newtheorem{definition}[definition]{Definition}
\aliascntresetthe{definition}

\newaliascnt{remark}{theorem}
\newtheorem{remark}[remark]{Remark}
\aliascntresetthe{remark}

\newaliascnt{example}{theorem}
\newtheorem{example}[example]{Example}
\aliascntresetthe{example}

\crefname{theorem}{theorem}{theorems}
\Crefname{theorem}{Theorem}{Theorems}
\crefname{proposition}{proposition}{propositions}
\Crefname{proposition}{Proposition}{Propositions}
\crefname{lemma}{lemma}{lemmas}
\Crefname{lemma}{Lemma}{Lemmas}
\crefname{corollary}{corollary}{corollaries}
\Crefname{corollary}{Corollary}{Corollaries}
\crefname{definition}{definition}{definitions}
\Crefname{definition}{Definition}{Definitions}
\crefname{remark}{remark}{remarks}
\Crefname{remark}{Remark}{Remarks}
\crefname{equation}{equation}{equations}
\Crefname{equation}{Equation}{Equations}
\crefname{section}{section}{sections}
\Crefname{section}{Section}{Sections}
\crefname{subsection}{subsection}{subsections}
\Crefname{subsection}{Subsection}{Subsections}

\newtheorem{letteredtheorem}{Theorem}

\newaliascnt{letteredcorollary}{letteredtheorem}
\newtheorem{letteredcorollary}[letteredcorollary]{Corollary}
\aliascntresetthe{letteredcorollary}

\newaliascnt{letteredproposition}{letteredtheorem}

\aliascntresetthe{letteredproposition}

\newaliascnt{letteredlemma}{letteredtheorem}

\aliascntresetthe{letteredlemma}

\crefname{letteredtheorem}{theorem}{theorems}
\Crefname{letteredtheorem}{Theorem}{Theorems}
\crefname{letteredcorollary}{corollary}{corollaries}
\Crefname{letteredcorollary}{Corollary}{Corollaries}
\crefname{letteredproposition}{proposition}{propositions}
\Crefname{letteredproposition}{Proposition}{Propositions}
\crefname{letteredlemma}{lemma}{lemmas}
\Crefname{letteredlemma}{Lemma}{Lemmas}

\newcommand{\R}{\mathbb R}
\newcommand{\C}{\mathbb C}

\newcommand{\id}{\mathrm{id}}
\newcommand{\loc}{\mathrm{loc}}

\newcommand{\scal}{\operatorname{scal}}
\newcommand{\Ric}{\operatorname{Ric}}

\newcommand{\Ahat}{\widehat{A}}
\newcommand{\dmu}{\,d\mu_g}

\DeclareMathOperator{\ind}{ind}
\DeclareMathOperator{\ch}{ch}

\title{Geometric rigidity via almost-harmonic twisted spinors}
\author{Francesco Bei and Simone Cecchini}
\thanks{SC was supported by a grant from the Simons Foundation (MPS-TSM-00007902, SC)}
\thanks{FB was partially supported by 2024 Sapienza research
grant {\em New research trends in Mathematics at Castelnuovo} and INdAM-GNSAGA ''Gruppo Nazionale per le Strutture Algebriche, Geometriche e le loro Applicazioni''}
\date{}

\begin{document}

\begin{abstract}
We establish sharp scalar-curvature bounds and rigidity consequences of
Gromov's exact-lift two-form method. Let \((M^n,g)\), \(n\geq 4\) even, be a
closed spin Riemannian manifold carrying a homologically
\(\Ahat\)-non-singular closed two-form \(\omega\) whose lift to the universal
cover \(X\) is exact. Then
\[
    \inf_M\scal_g
    \leq -\frac{4n}{n-1}\lambda_0(X).
\]
Equality forces \(g\) to be Einstein; if \(\lambda_0(X)>0\), then \(X\) is
real hyperbolic, while if \(\lambda_0(X)=0\) and
\(\int_M\omega^{n/2}\neq 0\), then \(g\) is flat.

The proof combines Gromov's twisted \(L^2\)-index with a conformal
interpretation of the refined Kato equality and a recentering argument. The
same method yields untwisted rigidity results when zero belongs to the
spectrum of the Dirac operator on the universal cover, with applications to
nonvanishing \(\widehat A\)-genus and enlargeability.
\end{abstract}

\maketitle

\begingroup
\linespread{0.90}\selectfont
\tableofcontents
\endgroup

\section{Introduction}
\label{sec:introduction}

The use of closed two-forms whose lifts to universal covers become exact is one of Gromov's influential ideas in global geometry.
In his work on K\"ahler hyperbolicity \cite{Gromov1991}, a closed K\"ahler manifold \((M,\omega)\) is
called K\"ahler hyperbolic if the lift of its K\"ahler form to the universal cover
is \(d\)-bounded, that is, if \(\pi^*\omega=d\eta\) for a bounded one-form
\(\eta\). Gromov showed that this condition has striking
\(L^2\)-Hodge-theoretic consequences: the \(L^2\)-cohomology of the universal
cover is concentrated in the middle degree, giving in particular the
Singer-type vanishing predicted for aspherical examples in this class, and
yielding the sign inequality \((-1)^m\chi(M)>0\) for compact K\"ahler hyperbolic manifolds of complex dimension \(m\).
More generally, if \(M^{2m}\) is closed and oriented, a closed two-form
\(\omega\in\Omega^2(M)\) is called homologically non-singular if
\[
        \int_M \omega^m\neq 0 .
\]
This condition records that the degree-two class \([\omega]\) carries non-trivial top-dimensional topology.
A central point in Gromov's work is that such degree-two topology, when combined with suitable control of the lifted form on the universal cover, can be converted into analytic information at infinity.
The method has since become an important tool in complex and Riemannian
geometry; see, for instance, \cite{JostZuo,CaoXavier}, and, for subsequent
developments concerning K\"ahler and weakly K\"ahler hyperbolic manifolds,
see \cite{BDET}.

The aim of the present paper is to study the implications of Gromov's two-form method for scalar curvature geometry.
One such implication appears in Gromov's quasisymplectic obstruction: if a closed orientable manifold carries a homologically non-singular two-form and its universal cover is contractible, then it admits no metric of positive scalar curvature \cite[Section 2.7]{GromovFourLectures}.
Since contractibility makes the lifted two-form exact, this obstruction suggests that topology detected by degree-two classes, when viewed on the universal cover, can impose strong constraints on scalar curvature.

In this paper, we show that exactness of the lifted two-form, without any
boundedness assumption on a primitive, converts this degree-two topology into
a sharp quantitative constraint on scalar curvature. Under a natural
\(\Ahat\)-refinement of homological non-singularity, we obtain the optimal
comparison between the infimum of the scalar curvature and the bottom of the spectrum of the
universal cover. We also establish rigidity at the sharp threshold: equality
forces the metric to be Einstein and, when the bottom of the spectrum is
positive, the universal cover to be real hyperbolic. In the top-degree case,
equality when the bottom of the spectrum is zero forces the metric to be flat.
The argument also yields a companion untwisted theorem, valid in every
dimension: whenever zero belongs to the \(L^2\)-spectrum of the spin Dirac
operator on the universal cover, the same sharp comparison holds, with
Einstein rigidity in the equality case and hyperbolic rigidity when the
bottom of the spectrum is positive. Both results arise from a common theory
of almost harmonic twisted spinors.

We use the following \(\Ahat\)-refinement of Gromov's homological non-singularity condition, adapted to spinorial techniques.
Let \(M^n\) be a closed connected oriented smooth manifold.
The \(\Ahat\)-class \(\Ahat(TM)\in H^{4*}(M;\mathbb Q)\) of \(TM\) is written as
\[
        \Ahat(TM)
        =1+\Ahat_1(TM)+\Ahat_2(TM)+\cdots,
\]
where \(\Ahat_q(TM)\in H^{4q}(M;\mathbb Q)\) is the component of degree \(4q\).
If \(g\) is a Riemannian metric on \(M\) with Levi-Civita connection \(\nabla^{TM}\), we denote by
\[
        \Ahat_q(\nabla^{TM})\in \Omega^{4q}(M)
\]
the corresponding Chern--Weil representative.

\begin{definition}
Let \(M^{2m}\) be a closed connected oriented smooth manifold. A closed
two-form \(\omega\in\Omega^2(M)\) is called homologically
\(\Ahat\)-non-singular if there exist integers \(p\geq 1\) and
\(q\geq 0\), with \(2p+4q=2m\), such that, for one (and hence every)
Riemannian metric \(g\) on \(M\),
\[
        \int_M \Ahat_q(\nabla^{TM})\wedge \omega^p \neq 0 .
\]
The usual homologically non-singular condition is the special case
\(p=m\), \(q=0\), namely
\[
        \int_M \omega^m\neq 0 .
\]
\end{definition}

The main geometric result is the following.

\begin{letteredtheorem}
\label{thm:main-two-form}
Let \((M^{2m},g_M)\), \(2m\geq 4\), be a closed connected spin Riemannian manifold, and let \(\pi:(X,g)\longrightarrow (M,g_M)\) be its universal Riemannian cover.
Suppose that \(M\) carries a closed homologically \(\Ahat\)-non-singular two-form \(\omega\in\Omega^2(M)\) whose lift to \(X\) is exact.
Let
\[
        \lambda_0=\lambda_0(X,g)
\]
be the bottom of the \(L^2\)-spectrum of the nonnegative scalar Laplacian on \((X,g)\).
Then
\begin{equation}\label{eq:scalar-bottom-estimate}
        \inf_M\scal_{g_M}
        \leq
        -\frac{4(2m)}{2m-1}\lambda_0.
\end{equation}
If equality holds in the previous inequality, then \(g\) is Einstein:
\[
        \Ric_g
        =
        -\frac{4\lambda_0}{2m-1}g.
\]
Furthermore, in the equality case, if \(\lambda_0>0\), then \((X,g)\) is isometric to real hyperbolic space with sectional curvature
\[
        \sec_g\equiv
        -\frac{4\lambda_0}{(2m-1)^2}.
\]
Finally, in the equality case, if \(\lambda_0=0\) and \(\omega\) is homologically non-singular in the top-degree sense,
\[
        \int_M\omega^m\neq 0,
\]
then \((X,g)\) is isometric to Euclidean space.
\end{letteredtheorem}

The same analytic framework, specialized to \(\omega=0\), yields a companion
untwisted rigidity theorem, valid in every dimension \(n\geq 3\).

\begin{letteredtheorem}
\label{thm:main-zero-in-the-spectrum}
Let \((M^n,g_M)\), \(n\geq 3\), be a closed connected spin Riemannian manifold, and let \(\pi:(X,g)\longrightarrow (M,g_M)\) be its universal Riemannian cover.
Let \(\lambda_0=\lambda_0(X,g)\) be the bottom of the \(L^2\)-spectrum of the nonnegative scalar Laplacian on \((X,g)\).
Suppose that \(0\) belongs to the \(L^2\)-spectrum of the spinor Dirac operator on \((X,g)\).
Then
\[
        \inf_M\scal_{g_M}
        \leq
        -\frac{4n}{n-1}\lambda_0.
\]
Moreover,
\begin{enumerate}[label=\textup{(\roman*)}]
    \item if \(\lambda_0=0\), then equality holds if and only if \(g_M\) is Ricci-flat;
    \item if \(\lambda_0>0\), then equality holds if and only if \((X,g)\) is isometric to real hyperbolic space with sectional curvature
    \[
        \sec_g\equiv
        -\frac{4\lambda_0}{(n-1)^2}.
    \]
\end{enumerate}
\end{letteredtheorem}

\begin{remark}
Wang--Zhu \cite{WangZhu2024} proved the sharp estimate and the corresponding
scalar curvature rigidity for rationally essential manifolds whose
fundamental group satisfies the strong Novikov conjecture. During the final
preparation of this manuscript, they posted the related paper
\cite{WangZhu2026}, in which they establish geometric rigidity at the same
sharp bottom-spectrum threshold under the topological assumptions of their
earlier work. The two works were obtained independently and offer
complementary approaches to rigidity. Wang--Zhu recover the missing Ricci
information directly from the limiting spinorial equality equations, while
our proof uses the conformal interpretation of the refined Kato defect to
construct a parallel spinor for a suitable conformally related metric.
\end{remark}

\begin{remark}
\Cref{thm:main-two-form} is not obtained from
\Cref{thm:main-zero-in-the-spectrum} simply by letting \(s\to 0\). Indeed,
if the primitive \(\eta\) were bounded, then
\[
        D_s-D_0=is\,c(\eta)
\]
would be a bounded perturbation whose operator norm tends to zero with
\(s\). Consequently, the normalized \(D_{s_j}\)-harmonic spinors furnished
by the central-extension index (see \Cref{subsec:Gromov-central-index})
would satisfy
\[
        \|D_0\psi_j\|_{L^2}
        \leq s_j\|\eta\|_{L^\infty}\|\psi_j\|_{L^2}
        \longrightarrow0,
\]
and hence would give \(0\in\sigma_{L^2}(D_0)\). Exactness of
\(\pi^*\omega\), however, does not provide a bounded primitive. For an
unbounded \(\eta\), Clifford multiplication by \(\eta\) is an unbounded
zeroth-order term; the domains of \(D_s\) and \(D_0\) need not agree, and
the preceding \(L^2\)-estimate is unavailable.
\end{remark}

An important source of the spectral hypothesis in
\Cref{thm:main-zero-in-the-spectrum} is compact enlargeability, in the
finite-cover sense recalled in \Cref{def:compact-enlargeability}. We prove in
\Cref{prop:compact-enlargeability-zero-spectrum} that compact enlargeability
and virtual spin imply that the spin Dirac operator on the universal cover has
zero in its \(L^2\)-spectrum. Closed oriented overtorical manifolds, namely
those admitting a map of nonzero degree to \(T^n\), provide a basic source of
compactly enlargeable examples; compare \cite[p.~210]{GromovLawson1980}.
Together with the Gromov--Lawson rigidity theorem
\cite[Theorem~A]{GromovLawson1980}, this gives the following consequence.

\begin{letteredcorollary}
\label{cor:compact-enlargeable-rigidity}
Let \((M^n,g_M)\), \(n\geq 3\), be a closed connected Riemannian manifold
whose underlying smooth manifold is compactly enlargeable and virtually
spin, and let
\(\pi:(X,g)\longrightarrow(M,g_M)\) be its universal Riemannian cover. Then
\[
        \inf_M\scal_{g_M}
        \leq
        -\frac{4n}{n-1}\lambda_0(X,g).
\]
If equality holds, then \(g_M\) has constant sectional curvature
\[
        \sec_{g_M}
        \equiv
        -\frac{4\lambda_0(X,g)}{(n-1)^2}.
\]
Thus an equality metric is flat when \(\lambda_0(X,g)=0\) and real
hyperbolic when \(\lambda_0(X,g)>0\).
\end{letteredcorollary}

\begin{remark}\label{rem:zero-in-the-spectrum-virtually-spin}
The spin assumptions in \Cref{thm:main-two-form,thm:main-zero-in-the-spectrum}
may be relaxed by assuming that \(M\) is virtually spin, with the spin
structure on \(X\) lifted from a finite spin cover. For
\Cref{thm:main-two-form}, the exact-lift condition and the homological
\(\Ahat\)-non-singularity condition pull back to finite covers.
\end{remark}

\begin{example}
The hypotheses of \Cref{thm:main-two-form} are satisfied by the following classes of manifolds:
\begin{itemize}
    \item closed spin aspherical manifolds admitting degree-two cohomology classes whose product is nonzero in top degree;
    \item closed spin symplectic manifolds \((M^{2m},\omega)\) whenever the symplectic form
    \(\omega\) has exact lift to the universal cover;
    \item four-dimensional closed spin manifolds for which \(\ker\bigl(\pi^*\colon H^2(M;\mathbb Q)\to H^2(X;\mathbb Q)\bigr)\) contains a class with nonzero square. Note that this happens whenever \(H^2(M;\mathbb Q)\neq \{0\}\) and \(H^2(X;\mathbb Q)=\{0\}\).
\end{itemize}
\end{example}

The inequality in \Cref{thm:main-two-form} is strict for every Riemannian
metric on several natural classes of manifolds. These include products of
higher-genus surfaces, certain symmetric products for which \(\lambda_0=0\),
and products such as \(K3\times B\), where the hypotheses are detected only
by the \(\Ahat\)-refinement. These three different obstructions to equality
are explained in
\Cref{ex:product-surfaces-strict,ex:symmetric-product-strict,ex:ahat-refined-strict}.

\begin{example}
Let \(M^n\), \(n\geq 3\), be a closed connected spin manifold.
Then \(M\) satisfies the hypotheses of \Cref{thm:main-zero-in-the-spectrum} in the following cases:
\begin{enumerate}[label=\textup{(\arabic*)}]
    \item \(M\) has nonvanishing reduced Rosenberg index; see
    \cite[Sections~1.5 and~2.1]{SchickCoarseIndex} and
    \cite[Proposition~2.5]{WangZhu2024}.
    \item The fundamental group of \(M\) satisfies the reduced Novikov conjecture and \(M\) carries a nonvanishing higher \(\widehat A\)-genus; see
    \cite[Theorem~2.11]{Rosenberg1983} and, for the reduced formulation used here,
    \cite[proof of Proposition~4.3]{WangZhu2024}.
    \item \(M\) has nonvanishing \(\widehat A\)-genus; see
    \cite[Proposition~3.2 and Remark~3.3]{Davaux2003}.
\end{enumerate}
\end{example}

\begin{example}
\Cref{cor:compact-enlargeable-rigidity} applies to every
Riemannian metric on a closed hyperbolic manifold, and equality characterizes
the appropriately normalized hyperbolic metrics. In dimension three, it
applies more generally to every closed oriented manifold which does not admit
a metric of positive scalar curvature. Precise statements and explanations
are given in \Cref{ex:closed-hyperbolic,ex:three-dimensional}.
\end{example}

\begin{remark}
An estimate of this type was first proved by Ono \cite{Ono1988}, who bounded
\(\inf_M\scal_{g_M}\) in terms of \(\lambda_0\) for closed spin manifolds
carrying certain index obstructions to positive scalar curvature. Davaux
\cite{Davaux2003} refined Ono's estimate to the optimal constant
\(4n/(n-1)\). Related sharp estimates were obtained by
Bei--Diverio--Trapani in the K\"ahler setting
\cite{FBeiDiverioTrapani}, and by Liu \cite{Liu26} via the vertical
\(\widehat A\)-cowaist. For complete three-manifolds, Munteanu--Wang
\cite[Theorems~1.1 and~1.3]{MunteanuWang2024} proved the sharp
bottom-spectrum estimate under scalar-curvature and Ricci-curvature lower
bounds and suitable topological hypotheses, and studied the corresponding
rigidity phenomena. In subsequent work, they removed the Ricci-curvature
lower bound from the spectral estimate for manifolds with finitely many ends
and finite first Betti number
\cite[Theorem~1.1]{MunteanuWang2026}.
\end{remark}

\begin{remark}\label{rem:hanke-schick}
By \eqref{eq:scalar-bottom-estimate}, the cohomological conditions of
\Cref{thm:main-two-form} obstruct the existence of positive scalar curvature
metrics on \(M\). Such cohomological obstructions should be viewed as
belonging to the broad circle of ideas surrounding the Novikov conjecture for
low-degree cohomology; see Connes--Gromov--Moscovici
\cite{ConnesGromovMoscovici}, Mathai \cite{MathaiLowDegree}, and
Hanke--Schick \cite{HankeSchick}. More precisely, by
\cite[Lemma 2.9]{BeiDiverioTrapani} and \cite{HankeSchick}, a homologically
\(\Ahat\)-non-singular two-form with exact lift implies nonvanishing of the
maximal Rosenberg index of \(M\). Hence \(M\) does not carry any metric of
positive scalar curvature. Related positive-scalar-curvature obstructions
arising from degree-two cohomology were also established by Guo--Mathai
\cite{GuoMathai} in the setting of proper actions. Finally, we mention the
recent work of Di Cerbo--Dranishnikov--Jauhari, which gives a complementary
positive-scalar-curvature obstruction for closed virtually spin
\(c\)-symplectically aspherical manifolds
\cite[Theorem~5.3]{DiCerboDranishnikovJauhari2026}.
\end{remark}

The analytic step that upgrades scalar rigidity to Einstein rigidity is isolated in the following theorem.

\begin{letteredtheorem}
\label{thm:intro-analytic-einstein}
Let \((M^n,g_M)\), \(n\geq 3\), be a closed connected spin Riemannian manifold, and let \(\pi:(X,g)\longrightarrow (M,g_M)\) be its universal Riemannian cover.
Let \(\omega\in\Omega^2(M)\) be a closed two-form whose lift is exact, that is,
\[
        \pi^*\omega=d\eta
\]
for some real one-form \(\eta\in\Omega^1(X)\). For \(s\in[0,1]\), let \(D_s\) be the Dirac operator on \(\Sigma_gX\otimes (X\times\C)\), where the trivial Hermitian line bundle \(X\times\C\) is equipped with the unitary connection \(d+is\eta\).

Assume that
\[
        \scal_g\equiv -\frac{4n}{n-1}\lambda_0(X,g).
\]
Assume moreover that there exist \(s_j\in[0,1]\), \(s_j\to0\), and nonzero smooth \(L^2\)-spinors
\[
        \psi_j\in C^\infty(X,\Sigma_gX\otimes (X\times\C))
        \cap L^2(X,\Sigma_gX\otimes (X\times\C))
\]
such that
\[
        \|\psi_j\|_{L^2}=1,
        \qquad
        \|D_{s_j}\psi_j\|_{L^2}\to0.
\]
Then
\[
        \Ric_g=-\frac{4\lambda_0(X,g)}{n-1}g.
\]
In particular, \(g\) is Einstein. Furthermore, if \(\lambda_0(X,g)>0\),
then \((X,g)\) is isometric to real hyperbolic space with constant sectional
curvature
\[
        \sec_g\equiv-\frac{4\lambda_0(X,g)}{(n-1)^2}.
\]
\end{letteredtheorem}

The sharp inequality itself comes from a Rayleigh--Kato--Lichnerowicz mechanism (see \Cref{subsec:sharp-scalar-upper-bound}) and builds on work of Davaux \cite{Davaux2003}. The Lichnerowicz formula \eqref{eq:twisted-lichnerowicz} relates the square of the Dirac operator to the spinorial connection and scalar curvature;
the refined Kato inequality \eqref{eq:refined-kato-inequality-harmonic} converts the spinorial energy into the energy of the
length function; and Rayleigh's characterization of the bottom of the spectrum \eqref{eq:rayleigh-characterization} then gives the
sharp scalar bound. To prove scalar rigidity, we use a Wang--Zhu-type positivity principle
\cite{WangZhu2024}, which we prove directly in the specific case needed here
(\Cref{prop:wang-zhu-positivity}). We apply it to the lengths of the almost
harmonic twisted spinors supplied by the hypotheses of
\Cref{thm:main-two-form,thm:main-zero-in-the-spectrum}.

Thus, once equality holds and scalar rigidity has forced
\[
        \scal_g
        \equiv
        -\frac{4n}{n-1}\lambda_0(X,g),
\]
the remaining question is how to interpret equality in this chain.

The proof of \Cref{thm:intro-analytic-einstein} gives a conformal interpretation of this equality case. The central object is the Kato defect, which measures the failure of a spinor to realize equality in the refined Kato inequality (\Cref{prop:refined-kato-identity}).
This defect is not merely an analytic error term.
Through the standard conformal transformation law for spinor connections, it is exactly the obstruction to making the rescaled spinor parallel for the conformally related metric (\Cref{lem:conformal-spinor-identity}).

This conformal spinorial viewpoint is related in spirit to the classical equality theory for Dirac eigenvalue estimates and twistor spinors, especially Hijazi's conformal refinement of Friedrich's estimate, Friedrich--Kim's analysis of its limiting cases, and Friedrich's work on the conformal relation between twistor and Killing spinors \cite{Hijazi1986,FriedrichKim2001,Friedrich1989}. The present setting differs in that it involves \(L^2\)-almost harmonic twisted spinors on a possibly noncompact universal cover and requires a delicate recentering argument to prevent loss of mass.

In the ideal situation where one has a single nonzero harmonic spinor \(\psi\) for which all equality defects vanish, this observation immediately produces the desired geometry.
Indeed, with \(r=|\psi|\) and
\[
        \widehat g = r^{4/(n-1)}g,
\]
the conformal spinor formula produces a nonzero parallel spinor for \(\widehat g\), which is therefore Ricci-flat.
Since the length \(r\) is then a generalized ground state for \(\Delta_g\), that is, \(\Delta_g r=\lambda_0 r\), a conformal Einstein criterion (\Cref{prop:conformal-einstein}) converts the Ricci-flatness of \(\widehat g\) into the Einstein equation for the original metric \(g\).
In the present argument one does not have such a single exact equality spinor.
Instead, one has a sequence of almost harmonic twisted spinors, with
\(s_j\to0\), for which both the Dirac defect and the equality defects vanish
asymptotically. In the two-form application, the \(\Gamma_{s_j}\)-\(L^2\)-index
produces such a sequence with \(D_{s_j}\psi_j=0\), whereas in the untwisted
case it is supplied by the assumption that zero lies in the spectrum of the
lifted Dirac operator.
The ground-state transform replaces the varying lengths \(|\psi_j|\) by a fixed generalized ground state, and a recentering argument using deck transformations prevents the \(L^2\)-mass from escaping to infinity.
After passing to a limit, the twisting disappears, and one obtains a positive generalized ground state \(f_\infty\) such that \(f_\infty^{4/(n-1)}g\) carries a nonzero parallel spinor.
The conformal Einstein criterion then shows that \(g\) is Einstein.

We finally comment on the sectional rigidity statement when \(\lambda_0>0\).
In the classical Ricci lower-bound setting, Cheng's comparison theorem \cite[Theorem~4.2]{Cheng1975} gives
\[
        \lambda_0(X,g)\leq \frac{(n-1)^2}{4}
\]
under the hypothesis \(\Ric_g\geq -(n-1)g\).\footnote{Cheng's
comparison theorem applies to complete manifolds in general, not only to
universal covers.}
Wang \cite[Theorem~1.4]{Wang2008HyperbolicSpace} proved that equality for the universal cover of a closed manifold forces real hyperbolicity.
Ledrappier--Wang \cite{LedrappierWang} obtained a related rigidity theorem for regular covers.
In both approaches, an entropy argument is used to produce a positive harmonic function \(\varphi\) satisfying the sharp identity
\[
        |\nabla\log\varphi|=n-1.
\]
Once such a function is available, the remaining step is the Li--Wang equality mechanism for the sharp Cheng--Yau gradient estimate (see \cite{LiWangPositiveSpectrumII}), in the warped-product form recorded by Wang \cite[Lemma~2]{Wang2008HyperbolicSpace} and in the special harmonic-function rigidity theorem of Ledrappier--Wang \cite[Theorem~6]{LedrappierWang}.

In the present paper, the special harmonic function is produced instead by the conformal spinorial compactness argument.
Thus the hyperbolic rigidity step itself belongs to the circle of ideas developed by Wang and Ledrappier--Wang, while the new input is the spinorial construction of the special harmonic function from the hypotheses of \Cref{thm:intro-analytic-einstein} (see the proof of \Cref{thm:conformal-geometric-rigidity}).

The paper is organized as follows. In \Cref{sec:preliminaries} we collect
the spectral and spinorial preliminaries, including the bottom of the
\(L^2\)-spectrum, twisted Dirac operators, and the equivariant
\(L^2\)-index theorem. In
\Cref{sec:almost-harmonic-twisted-spinors} we introduce almost harmonic
twisted spinors and establish the two mechanisms that produce them: the
zero-in-the-spectrum mechanism and Gromov's central-extension
\(L^2\)-index mechanism; we also collect the examples and applications
described above. In \Cref{sec:scalar-rigidity} we prove the sharp
scalar-curvature comparison and scalar rigidity. Finally, in
\Cref{sec:einstein-rigidity} we prove
\Cref{thm:intro-analytic-einstein}, upgrade Einstein rigidity to constant
sectional curvature in the positive-spectrum case, and derive
\Cref{thm:main-two-form,thm:main-zero-in-the-spectrum} together with
\Cref{cor:compact-enlargeable-rigidity}.

\subsection*{Acknowledgments}

We thank Simon Brendle for helpful comments on the refined Kato inequality,
Guoliang Yu for pointing out the connection with the work of Hanke and Schick
(see \Cref{rem:hanke-schick}), and Christian B\"ar for interesting comments
during a presentation of an earlier version of this work.

\section{Spectral and spinorial preliminaries}
\label{sec:preliminaries}

\subsection{The bottom of the \texorpdfstring{\(L^2\)}{L2}-spectrum of the Laplacian}
\label{subsec:bottom-spectrum}

Unless explicitly stated otherwise, in the two-form index arguments we write
\(n=2m\geq 4\). The untwisted and ungraded analytic statements below are
valid in the dimensions specified locally.
Let \((M^n,g_M)\) be a closed connected Riemannian manifold, and let
\[
        \pi:X=\widetilde M\longrightarrow M
\]
be the universal cover. We denote by \(g=\pi^*g_M\) the lifted metric on \(X\). Since \(M\) is closed, \((X,g)\) is complete.

We regard the nonnegative scalar Laplacian
\[
        \Delta_g=d^*d
\]
as an unbounded operator on \(L^2(X,d\mu_g)\), initially defined on \(C_c^\infty(X)\). Since \((X,g)\) is complete, \(\Delta_g\) is essentially self-adjoint on \(C_c^\infty(X)\) by Chernoff's theorem \cite{ChernoffEssentialSelfAdjointness}. We denote its self-adjoint \(L^2\)-closure again by \(\Delta_g\), and write \(\sigma_{L^2}(\Delta_g)\) for its spectrum.

We write
\[
        \lambda_0=\lambda_0(X,g)
\]
for the bottom of the \(L^2\)-spectrum of \(\Delta_g\), namely
\[
        \lambda_0:=\inf\sigma_{L^2}(\Delta_g).
\]
Equivalently, by the Rayleigh characterization,
\begin{equation}
\label{eq:rayleigh-characterization}
        \lambda_0
        =
        \inf_{0\ne u\in C_c^\infty(X)}
        \frac{\int_X |du|^2\dmu}{\int_X |u|^2\dmu}.
\end{equation}
For this standard material on the \(L^2\)-spectrum of the Laplacian on complete Riemannian manifolds, see for instance \cite{ChavelEigenvalues,StrichartzLaplacian}.

The closed quadratic form associated with \(\Delta_g\) is
\[
        Q[u]=\int_X |du|^2\dmu,
\]
with form domain \(W^{1,2}(X)\). Since \((X,g)\) is complete, \(C_c^\infty(X)\) is a form core for \(Q\). Hence \eqref{eq:rayleigh-characterization} can equivalently be written as
\[
        \lambda_0
        =
        \inf_{0\ne u\in W^{1,2}(X)}
        \frac{\int_X |du|^2\dmu}{\int_X |u|^2\dmu}.
\]
Consequently,
\begin{equation}
\label{eq:rayleigh-inequality}
        \int_X |du|^2\dmu
        \ge
        \lambda_0\int_X |u|^2\dmu,
        \qquad
        u\in W^{1,2}(X).
\end{equation}

\subsection{Twisted Dirac operators}
\label{subsec:twisted-dirac}

Assume now that \(M\) is spin. We fix a spin structure on \(M\), and we equip \(X\) with the pulled-back spin structure. 
Thus \(\Sigma_gX=\pi^*\Sigma_{g_M}M\).
Let \(\nabla^{\Sigma_gX}\) be the pulled-back spinor connection. 
The Clifford multiplication for \(g\) is denoted by \(c_g\), or simply by \(c\) when no confusion is possible. 
Our Clifford convention is
\[
        c(v)^*=-c(v),
        \qquad
        c(v)^2=-|v|_g^2.
\]

Let \(E\to X\) be a Hermitian vector bundle equipped with a unitary connection \(\nabla^E\). We write
\[
        \nabla^{\Sigma\otimes E}
        :=
        \nabla^{\Sigma_gX}\otimes 1+1\otimes\nabla^E
\]
for the induced connection on \(\Sigma_gX\otimes E\). The corresponding twisted Dirac operator is
\[
        D_E:C_c^\infty(X,\Sigma_gX\otimes E)
        \longrightarrow
        C_c^\infty(X,\Sigma_gX\otimes E),
        \qquad
        D_E\psi=\sum_{a=1}^n c(e_a)\nabla^{\Sigma\otimes E}_{e_a}\psi,
\]
where \((e_a)\) is any local \(g\)-orthonormal frame. On \(C_c^\infty\), the operator \(D_E\) is symmetric. Since \((X,g)\) is complete, \(D_E\) is essentially self-adjoint by Chernoff's theorem \cite{ChernoffEssentialSelfAdjointness}. We denote its self-adjoint \(L^2\)-closure again by \(D_E\).

The Lichnerowicz formula for \(D_E\) is
\begin{equation}
\label{eq:general-twisted-lichnerowicz}
        D_E^2
        =
        (\nabla^{\Sigma\otimes E})^*\nabla^{\Sigma\otimes E}
        +\frac14\scal_g+\mathcal R^E,
\end{equation}
where the twisting curvature term is
\[
        \mathcal R^E
        =
        \frac12\sum_{a,b=1}^n c(e_a)c(e_b)R^E(e_a,e_b)
        \in
        \operatorname{End}(\Sigma_gX\otimes E).
\]
See, for instance, \cite[Theorem~II.8.17]{LawsonMichelsohn}.

We now specialize this notation to the line-bundle twists used below. Let \(\omega\in\Omega^2(M)\) be a closed two-form whose lift to \(X\) is exact:
\[
        \pi^*\omega=d\eta
\]
for some real one-form \(\eta\in\Omega^1(X)\). Let
\[
        F:=X\times\C
\]
be the trivial Hermitian line bundle. For \(s\in\R\), we equip \(F\) with the unitary connection
\[
        \nabla^{F,s}:=d+is\eta.
\]
We write
\[
        \nabla^s
        :=
        \nabla^{\Sigma_gX}\otimes1+1\otimes\nabla^{F,s}
\]
for the induced connection on \(\Sigma_gX\otimes F\). The corresponding twisted Dirac operator is
\[
        D_s:C_c^\infty(X,\Sigma_gX\otimes F)
        \longrightarrow
        C_c^\infty(X,\Sigma_gX\otimes F),
        \qquad
        D_s\psi=\sum_{a=1}^n c(e_a)\nabla^s_{e_a}\psi,
\]
where \((e_a)\) is any local \(g\)-orthonormal frame. On \(C_c^\infty\), the operator \(D_s\) is symmetric. Since \((X,g)\) is complete, \(D_s\) is essentially self-adjoint by Chernoff's theorem \cite{ChernoffEssentialSelfAdjointness}. We denote its self-adjoint \(L^2\)-closure again by \(D_s\).

The curvature of \(\nabla^{F,s}\) is
\begin{equation}
\label{eq:line-curvature}
        F(\nabla^{F,s})=is\,d\eta=is\,\pi^*\omega.
\end{equation}
Hence the twisted Lichnerowicz formula has the form
\begin{equation}
\label{eq:twisted-lichnerowicz}
        D_s^2=(\nabla^s)^*\nabla^s+\frac14\scal_g+sR_\omega,
\end{equation}
where \(R_\omega\in\operatorname{End}(\Sigma_gX\otimes F)\) is defined by
\[
        R_\omega
        =
        \frac i2\sum_{a,b=1}^n
        c(e_a)c(e_b)\,\pi^*\omega(e_a,e_b).
\]

Since \(\omega\) is smooth on the compact manifold \(M\), its lift \(\pi^*\omega\) is bounded on \(X\). 
Therefore there is a constant \(C_\omega>0\), depending only on \(n\), \(g_M\), and \(\omega\), such that
\begin{equation}
\label{eq:Romega-bound}
        |\langle R_\omega\varphi,\varphi\rangle|
        \le
        C_\omega|\varphi|^2
\end{equation}
for every twisted spinor \(\varphi\in C^\infty(X,\Sigma_gX\otimes F)\).

We shall use the following integrated form of the Lichnerowicz formula. Suppose that
\[
        \psi\in C^\infty(X,\Sigma_gX\otimes F)
        \cap L^2(X,\Sigma_gX\otimes F)
\]
and that \(D_s\psi\in L^2(X,\Sigma_gX\otimes F)\). Then
\[
        \nabla^s\psi\in L^2(X,T^*X\otimes\Sigma_gX\otimes F),
\]
and
\begin{equation}
\label{eq:integrated-lichnerowicz}
        \|D_s\psi\|_{L^2}^2
        =
        \|\nabla^s\psi\|_{L^2}^2
        +\frac14\int_X \scal_g|\psi|^2\dmu
        +s\int_X\langle R_\omega\psi,\psi\rangle\dmu.
\end{equation}

\subsection{The \texorpdfstring{\(L^2\)-\(G\)}{L2-G}-index theorem}
\label{subsec:equivariant-L2-index}

We recall the \(L^2\)-\(G\)-index and the equivariant index theorem used
below. This is the extension of Atiyah's \(L^2\)-index theorem
\cite{AtiyahL2Index} considered by Ballmann in
\cite[Section~8.2]{Bal06}; it applies even when the Dirac bundle on the
covering does not descend to the quotient.

Let
\[
        p:\overline M\longrightarrow M
\]
be a normal Riemannian covering, where \(M\) and \(\overline M\) are
connected and \(M\) is closed. Let \(\Gamma\) be its deck group, and let
\begin{equation}
\label{eq:compact-group-extension}
        1\longrightarrow K\longrightarrow G
        \xrightarrow{\rho}\Gamma\longrightarrow1
\end{equation}
be an extension of \(\Gamma\) by a compact group \(K\). Let
\(\overline V=\overline V^+\oplus\overline V^-\to\overline M\) be a
graded Hermitian Dirac bundle, not necessarily the pullback of a bundle
on \(M\). Suppose that \(G\) acts on \(\overline V\) by Hermitian bundle
isomorphisms, with \(g\in G\) covering the deck transformation
\(\rho(g)\). Equivalently, if \(q:\overline V\to\overline M\) denotes
the bundle projection, then
\[
        q\circ g=\rho(g)\circ q.
\]
Assume that the action preserves the grading, the Dirac connection, and
Clifford multiplication. The associated Dirac operator
\[
\overline D^+:C^\infty(\overline M,\overline V^+)
        \longrightarrow C^\infty(\overline M,\overline V^-)
\]
is then \(G\)-equivariant. Let
\[
        \overline D^-=(\overline D^+)^*.
\]
Explicitly, the induced action on sections is
\begin{equation}
\label{eq:G-action-on-sections}
        (g\cdot\sigma)(x)
        :=
        g_{\rho(g)^{-1}x}
        \bigl(\sigma(\rho(g)^{-1}x)\bigr).
\end{equation}

Set
\[
        H^\pm:=L^2(\overline M,\overline V^\pm),
\]
and let
\[
        P^\pm:H^\pm\longrightarrow\ker\overline D^\pm
\]
be the orthogonal projections. Since \(\overline D^\pm\) are
\(G\)-equivariant, the projections \(P^\pm\) commute with the \(G\)-action.
By elliptic regularity they are smoothing operators. Let
\[
        \mathcal K^\pm(x,y)\in
        \operatorname{Hom}
        \bigl(\overline V^\pm_y,\overline V^\pm_x\bigr)
\]
denote their smooth Schwartz kernels.

If \(g\in G\) maps to \(\gamma=\rho(g)\in\Gamma\), then
\begin{equation}
\label{eq:G-kernel-equivariance}
        \mathcal K^\pm(\gamma x,\gamma y)
        =
        g_x\mathcal K^\pm(x,y)g_y^{-1}.
\end{equation}
In particular,
\[
        \operatorname{tr}\mathcal K^\pm(\gamma x,\gamma x)
        =
        \operatorname{tr}\mathcal K^\pm(x,x),
\]
so \(x\mapsto\operatorname{tr}\mathcal K^\pm(x,x)\) is a
\(\Gamma\)-invariant function on \(\overline M\).

Let \(\mathcal F\subset\overline M\) be a measurable fundamental domain.
The \(G\)-dimensions of the two kernels are defined by
\begin{equation}
\label{eq:extended-dimension}
        \dim_G\ker\overline D^\pm
        :=\int_{\mathcal F}\operatorname{tr}\mathcal K^\pm(x,x)\,d\mu(x)
\end{equation}
and the \(L^2\)-\(G\)-index of \(\overline D^+\) is
\begin{equation}
\label{eq:extended-L2-index-definition}
        \ind_G^{(2)}(\overline D^+)
        :=
        \dim_G\ker\overline D^+
        -
        \dim_G\ker\overline D^-.
\end{equation}
The integrals in \eqref{eq:extended-dimension} are independent of the
choice of fundamental domain. Their finiteness, and hence the
well-definedness of \eqref{eq:extended-L2-index-definition}, follow from
the next theorem.

\begin{theorem}
\label{thm:compact-extension-L2-index}
In the preceding setting, the associated Dirac operator is
\(G\)-equivariant and the \(G\)-dimensions
\[
        \dim_G\ker\overline D^+,
        \qquad
        \dim_G\ker\overline D^-
\]
are finite. Moreover, the canonical local index form of
\(\overline D^+\) is the pullback of a differential form
\(\alpha_{\overline D}\) on \(M\), and
\begin{equation}
\label{eq:compact-extension-L2-index}
        \ind_G^{(2)}(\overline D^+)
        =\int_M\alpha_{\overline D}.
\end{equation}
\end{theorem}

\begin{remark}
The case of \Cref{thm:compact-extension-L2-index} used in
\Cref{subsec:Gromov-central-index}, in which \(G\) is a \(U(1)\)-central
extension of \(\Gamma\), follows from Hang Wang's \(L^2\)-index theorem for
proper cocompact group actions
\cite[Proposition~3.6 and Theorem~6.10]{WangProperCocompactIndex}.
\end{remark}

\begin{corollary}
\label{cor:nonzero-index-L2-kernel}
If \(\int_M\alpha_{\overline D}\ne0\), then the full operator
\(\overline D=\overline D^++\overline D^-\) has a nonzero smooth
\(L^2\)-kernel.  In particular,
\(0\in\sigma_{L^2}(\overline D)\). In the pullback case
\(\overline V=p^*V\), where \(D^+\) denotes the quotient operator on \(M\),
the same conclusion holds whenever \(\ind(D^+)\ne0\).
\end{corollary}

\begin{proof}
By the preceding theorem, at least one of
\(\dim_G\ker\overline D^+\) and
\(\dim_G\ker\overline D^-\) is nonzero.  Hence the
corresponding Hilbert-space kernel is nonzero.  Elliptic regularity makes
its elements smooth.
\end{proof}

The extra generality of \Cref{thm:compact-extension-L2-index} is essential
for Gromov's construction: see \Cref{subsec:Gromov-central-index}.

\section{Almost harmonic twisted spinors}
\label{sec:almost-harmonic-twisted-spinors}

We keep the notation of \Cref{subsec:twisted-dirac}. Thus
\(F=X\times\C\) is equipped with the family of unitary connections
\[
        \nabla^{F,s}=d+is\eta,
        \qquad s\in\R,
\]
and \(D_s\) denotes the corresponding twisted Dirac operator on
\(\Sigma_gX\otimes F\). We isolate the analytic input used in the
scalar-curvature comparison argument in the following definition.

\subsection{The almost-harmonic condition}
\label{subsec:almost-harmonic-definition}

\begin{definition}
\label{def:almost-harmonic-twisted-spinors}
We say that the data \((X,g,\eta)\) admit almost harmonic twisted spinors,
asymptotically as \(s\to0\), if there exist a sequence \((s_j)_{j\geq1}\subset[0,1]\), with \(s_j\longrightarrow0\), and smooth twisted spinors
\[
        \psi_j\in
        C^\infty(X,\Sigma_gX\otimes F)
        \cap
        L^2(X,\Sigma_gX\otimes F)
\]
such that
\[
        D_{s_j}\psi_j\in L^2(X,\Sigma_gX\otimes F),
        \qquad
        \|\psi_j\|_{L^2}=1,
        \qquad
        \|D_{s_j}\psi_j\|_{L^2}\longrightarrow0.
\]
\end{definition}

The case \(s_j=0\) for every \(j\) is allowed. In that case one recovers
the usual untwisted \(L^2\)-almost harmonic spinors for the lifted spin
Dirac operator \(D_0=D_X\). Notice also that the twisting-curvature term in
the Lichnerowicz formula is \(s_jR_\omega\), and therefore it is
automatically asymptotically negligible by \eqref{eq:Romega-bound}.

\begin{remark}
\label{rem:almost-harmonic-independent-primitive}
Suppose that \(\eta'\in\Omega^1(X)\) is another primitive of
\(\pi^*\omega\). Since \(X\) is simply connected, there exists
\(\phi\in C^\infty(X,\R)\) such that \(\eta'=\eta+d\phi\).
Let \(D_s^\eta\) and \(D_s^{\eta'}\) be the twisted Dirac operators
associated with the connections \(d+is\eta\) and \(d+is\eta'\),
respectively. For \(s\in\R\), multiplication by \(e^{-is\phi}\) defines a
fiberwise unitary bundle automorphism
\[
        G_s:\Sigma_gX\otimes F\longrightarrow\Sigma_gX\otimes F.
\]
A direct computation gives
\[
        D_s^{\eta'}G_s=G_sD_s^\eta
\]
on \(C_c^\infty(X,\Sigma_gX\otimes F)\), and hence also for the
self-adjoint \(L^2\)-closures. Thus, if \((s_j,\psi_j)\) is an almost
harmonic sequence for \((X,g,\eta)\), then
\[
        \psi_j':=G_{s_j}\psi_j
\]
is an almost harmonic sequence for \((X,g,\eta')\). In particular, the
existence of almost harmonic twisted spinors is independent of the chosen
primitive.
\end{remark}

\begin{remark}
\label{rem:almost-harmonic-two-cases}
The constructions below produce the sequences in
\Cref{def:almost-harmonic-twisted-spinors} in two ways.
In the untwisted case, \(s_j=0\) for every \(j\), and
\(0\in\sigma_{L^2}(D_X)\) yields an almost harmonic sequence by the
spectral theorem. In the genuinely twisted case, one chooses
\(s_j\in(0,1]\) with \(s_j\to0\); the nonvanishing
\(\Gamma_{s_j}\)-\(L^2\)-index then produces normalized \(L^2\)-harmonic
twisted spinors satisfying \(D_{s_j}\psi_j=0\). The first mechanism is
treated in \Cref{subsec:zero-in-the-spectrum}, and the second in
\Cref{subsec:Gromov-central-index}.
\end{remark}

\subsection{The zero-in-the-spectrum mechanism}
\label{subsec:zero-in-the-spectrum}

The first mechanism is the direct, untwisted spectral one.  In the notation
of \Cref{subsec:twisted-dirac}, it is obtained by taking
\(s_j=0\) for every \(j\).  Equivalently, for this mechanism one may take
\(\omega=0\) and \(\eta=0\), so that \(D_s=D_X\) for every \(s\).  Thus the
auxiliary line is the trivial flat line and the curvature error vanishes
identically.

\begin{proposition}
\label{prop:zero-in-spectrum-produces-almost-harmonic}
Let \((X,g)\) be a complete spin Riemannian manifold and let \(D_X\) be its
spin Dirac operator.  If \(0\in\sigma_{L^2}(D_X)\), then there exist spinors \(\psi_j\in C_c^\infty(X,\Sigma_gX)\) such that
\[
        \|\psi_j\|_{L^2}=1,
        \qquad
        \|D_X\psi_j\|_{L^2}\longrightarrow0.
\]
Equivalently, regarding \(\psi_j\) as spinors with coefficients in the
trivial flat line, they form an almost harmonic sequence with
\(s_j=0\) for every \(j\).
\end{proposition}

\begin{proof}
By Weyl's criterion \cite[Theorem~VII.12 and p.~264]{ReedSimonI}, for each
\(j\) there exists a vector \(\phi_j\) in the domain of \(D_X\) such that
\[
        \|\phi_j\|_{L^2}=1,
        \qquad
        \|D_X\phi_j\|_{L^2}<\frac1j.
\]
Since \((X,g)\) is complete, \(C_c^\infty(X,\Sigma_gX)\) is a core for
\(D_X\).  Approximating \(\phi_j\) in the graph norm and normalizing, we may
choose
\[
        \psi_j\in C_c^\infty(X,\Sigma_gX),
        \qquad
        \|\psi_j\|_{L^2}=1,
        \qquad
        \|D_X\psi_j\|_{L^2}\longrightarrow0.
\]
Since the \(\psi_j\) are compactly supported and smooth, their covariant
derivatives belong to \(L^2\).  For the trivial flat line the Lichnerowicz
curvature term is zero.  Hence this is an almost harmonic sequence in the
sense of \Cref{def:almost-harmonic-twisted-spinors}, with \(s_j=0\) for all
\(j\).
\end{proof}

We next recall the finite-cover form of enlargeability introduced by
Gromov--Lawson \cite[p.~210]{GromovLawson1980}. Following the terminology
of \cite{HankeSchickEnlargeability}, we call this property compact
enlargeability, to distinguish it from the version allowing infinite
covers. Gromov--Lawson included a spin condition in their original
formulation; here we separate that condition from compact enlargeability
and impose virtual spin explicitly when it is needed.

\begin{definition}
\label{def:compact-enlargeability}
A closed connected smooth manifold \(M^n\) is compactly enlargeable if,
for one, and hence every, Riemannian metric \(h\) on \(M\), the following
holds: for every \(\varepsilon>0\), there exist a finite connected
orientable covering \(p_\varepsilon:M_\varepsilon\longrightarrow M\) and a smooth map \(f_\varepsilon:M_\varepsilon\longrightarrow S^n\) of nonzero degree such that
\begin{equation}
\label{eq:epsilon-contracting}
        |df_\varepsilon(v)|_{S^n}
        \leq
        \varepsilon |v|_{p_\varepsilon^*h}
\end{equation}
for every \(v\in TM_\varepsilon\), where \(S^n\) carries the unit round
metric.
\end{definition}

The metric-independence follows because any two metrics on the compact
manifold \(M\) are uniformly equivalent, with the same comparison
constants on all their lifts.

\begin{proposition}
\label{prop:compact-enlargeability-zero-spectrum}
Let \(M^n\) be a closed connected compactly enlargeable virtually spin
manifold, let \(g_M\) be any Riemannian metric on \(M\), and let
\((X,g)\to(M,g_M)\) be the universal Riemannian cover, equipped with the
spin structure lifted from a finite spin cover of \(M\). Then its spin
Dirac operator satisfies
\[
        0\in\sigma_{L^2}(D_X).
\]
\end{proposition}

\begin{proof}
We first reduce to the spin case. Let \(q:M'\to M\) be a finite spin
cover. Compact enlargeability passes to \(M'\): given an enlarging cover
\(p_\varepsilon:M_\varepsilon\to M\), choose a connected component
\(N_\varepsilon\) of the fiber product \(M'\times_M M_\varepsilon\).
The two projections make \(N_\varepsilon\) a finite cover of both \(M'\) and \(M_\varepsilon\). 
The composite
\[
        N_\varepsilon\longrightarrow M_\varepsilon
        \xrightarrow{f_\varepsilon}S^n
\]
is still \(\varepsilon\)-contracting and has degree
\(\deg(N_\varepsilon\to M_\varepsilon)\deg(f_\varepsilon)\ne0\).
Therefore \(M'\) is compactly enlargeable. Its universal Riemannian cover
is again \((X,g)\), with the spin structure pulled back from \(M'\). We
may consequently assume from now on that \(M\) is spin.

Suppose first that \(n=2m\). If \(\widehat A(M)\ne0\), the conclusion
follows from \cite[Proposition~3.2 and Remark~3.3]{Davaux2003}. Let us
therefore assume that
\[
        \widehat A(M)=0.
\]
For every finite covering \(p_\varepsilon:M_\varepsilon\to M\), multiplicativity under finite covers gives \(\widehat A(M_\varepsilon)=0\).
For every \(\varepsilon>0\), choose
\(p_\varepsilon:M_\varepsilon\to M\) and
\(f_\varepsilon:M_\varepsilon\to S^n\) as in
\Cref{def:compact-enlargeability}. The spin structure of \(M\) pulls back
to \(M_\varepsilon\).

Choose a Hermitian bundle \(E_0\to S^n\) whose reduced \(K\)-theory class
is a nonzero Bott class, and choose a Hermitian complement \(F_0\) in a
trivial bundle. Endow both bundles with unitary connections. Then
\begin{equation}
\label{eq:bott-complementary-bundles}
        E_0\oplus F_0\cong S^n\times\C^k,
        \qquad
        \left\langle
        \ch(E_0),[S^n]
        \right\rangle\ne0.
\end{equation}
Set
\[
        E_\varepsilon=f_\varepsilon^*E_0,
        \qquad
        F_\varepsilon=f_\varepsilon^*F_0.
\]
Davaux's enlargeability argument \cite[Section~6.1]{Davaux2003} gives a
nonzero smooth \(L^2\)-harmonic spinor with coefficients in the lift
\(\widetilde E_\varepsilon\) of \(E_\varepsilon\) to \(X\). Choose it
normalized and denote it by
\[
        \psi_\varepsilon^E
        \in L^2(X,\Sigma_gX\otimes\widetilde E_\varepsilon),
        \qquad
        \|\psi_\varepsilon^E\|_{L^2}=1,
        \qquad
        D_{\widetilde E_\varepsilon}\psi_\varepsilon^E=0.
\]
Fix the trivialization in \eqref{eq:bott-complementary-bundles}, and write
the direct-sum connection on \(E_0\oplus F_0\) as \(d+A_0\). If
\(\widetilde f_\varepsilon:X\to S^n\) denotes the composite of the
covering \(X\to M_\varepsilon\) with \(f_\varepsilon\), then the lifted
connection is
\[
        d+A_\varepsilon,
        \qquad
        A_\varepsilon:=\widetilde f_\varepsilon^*A_0.
\]
Under the induced identification
\[
        \widetilde E_\varepsilon\oplus\widetilde F_\varepsilon
        \cong X\times\C^k,
\]
let \(D_{A_\varepsilon}\) denote the twisted Dirac operator on
\(\Sigma_gX\otimes\C^k\) associated with the connection
\(d+A_\varepsilon\). The direct-sum connection preserves the two summands,
and hence
\[
        D_{A_\varepsilon}
        =
        D_{\widetilde E_\varepsilon}
        \oplus
        D_{\widetilde F_\varepsilon}.
\]
Therefore
\[
        \psi_\varepsilon:=(\psi_\varepsilon^E,0)
\]
satisfies
\[
        \psi_\varepsilon
        \in L^2(X,\Sigma_gX\otimes\C^k),
        \qquad
        \|\psi_\varepsilon\|_{L^2}=1,
        \qquad
        D_{A_\varepsilon}\psi_\varepsilon=0.
\]
Since the covering maps are local isometries and \(f_\varepsilon\) is
\(\varepsilon\)-contracting, there is a constant \(C>0\), independent of
\(\varepsilon\), such that
\begin{equation}
\label{eq:enlargeability-connection-bound}
        \|A_\varepsilon\|_{L^\infty(X)}
        \leq C\varepsilon.
\end{equation}
Relative to the trivial connection, this operator is the bounded perturbation
\[
        D_{A_\varepsilon}
        =
        D_X\otimes1_{\C^k}+c(A_\varepsilon).
\]
Since \(D_{A_\varepsilon}\psi_\varepsilon=0\) and
\(c(A_\varepsilon)\) is bounded, this identity implies that
\(\psi_\varepsilon\) belongs to the domain of
\(D_X\otimes1_{\C^k}\), and \eqref{eq:enlargeability-connection-bound}
gives
\[
        \bigl\|
        (D_X\otimes1_{\C^k})\psi_\varepsilon
        \bigr\|_{L^2}
        \leq C'\varepsilon.
\]
Letting \(\varepsilon\to0\) shows that
\(0\in\sigma_{L^2}(D_X\otimes1_{\C^k})\). This operator is a finite
orthogonal sum of copies of \(D_X\), so
\[
        0\in\sigma_{L^2}(D_X).
\]

Now suppose that \(n\) is odd. Since both \(M\) and \(S^1\) are compactly
enlargeable, the product property for enlargeable manifolds
\cite[p.~210]{GromovLawson1980} implies that \(M\times S^1\) is compactly
enlargeable. Since \(M\) is spin, so is \(M\times S^1\).

The even-dimensional case, applied to \(M\times S^1\) with the product
metric, now yields \(0\in\sigma_{L^2}(D_{X\times\R})\).
Under the standard product identification of spinor bundles, up to the
usual finite multiplicity,
\[
        D_{X\times\R}^2
        =
        D_X^2\otimes1+1\otimes D_\R^2
\]
as nonnegative self-adjoint operators. If
\(0\notin\sigma_{L^2}(D_X)\), then, for some \(a>0\),
\(D_X^2\geq a^2\). The product formula would imply
\(D_{X\times\R}^2\geq a^2\), contradicting
\(0\in\sigma_{L^2}(D_{X\times\R})\). 
Therefore \(0\in\sigma_{L^2}(D_X)\), which completes the proof.
\end{proof}

\subsection{Gromov's central extension and the twisted
\texorpdfstring{\(L^2\)}{L2}-index mechanism}
\label{subsec:Gromov-central-index}

We keep the notation of \Cref{subsec:twisted-dirac} and assume throughout
this subsection that \(n\) is even. We then have the chirality splitting
\[
\Sigma_gX=\Sigma_g^+X\oplus\Sigma_g^-X. 
\]
Moreover, the twisted Dirac operator is odd with respect to the corresponding splitting of \(\Sigma_gX\otimes F\):
\[
    D_s^\pm\colon C^\infty(X;\Sigma_gX^\pm\otimes F)\to C^\infty(X;\Sigma_gX^\mp\otimes F),
\]
where \(D_s^\pm=D_s|_{\Sigma_gX^\pm\otimes F}\).

Since the spin structure on \(X\) is pulled back from the fixed spin structure on \(M\), every deck transformation \(\gamma\in\Gamma\) has a canonical lift to the spinor bundle. 
Thus \(\Sigma_gX=\pi^*\Sigma_{g_M}M\) carries a unitary \(\Gamma\)-action covering the deck action.
This action preserves Clifford multiplication, the pulled-back spin connection, and the chirality splitting.

\begin{remark}
Our assumption that \(M\) is spin makes the lifted action of the deck group on
the spinor bundle canonical. If only the universal cover \(X\) is spin, one
must replace the deck group by Rosenberg's \(\{\pm1\}\)-extension
\cite[Section 3.B]{Rosenberg1983}. In the twisted setting, this extension must
be combined with Gromov's \(U(1)\)-central extension introduced below. We do
not pursue this more general formulation here.
\end{remark}

We recall Gromov's central extension \cite{Gromov1991} in the form used by Ballmann \cite[Section 8.2, Lemma 8.29]{Bal06}.
Let \(\Gamma_s\) be the group of unitary bundle automorphisms \(\Phi:F\longrightarrow F\) which cover deck transformations and preserve the connection \(\nabla^{F,s}\). 
If \(\Phi\) covers \(\gamma\in\Gamma\), we write
\[
\sigma_s(\Phi)=\gamma .
\]
Since \(X\) is simply connected and the curvature \(F(\nabla^{F,s})=is\pi^*\omega\) is \(\Gamma\)-invariant, by \cite[Lemma~8.29]{Bal06} every deck transformation admits such a connection-preserving lift. 
The kernel consists of constant unitary multiplications on \(F\). 
Hence one obtains a central extension
\[
1\longrightarrow U(1)\longrightarrow \Gamma_s
\stackrel{\sigma_s}{\longrightarrow}\Gamma\longrightarrow 1 .
\]
We refer to \cite[Section 8]{Bal06} for details.

If \(\Phi\in\Gamma_s\) covers \(\gamma=\sigma_s(\Phi)\), let
\[
U_\Phi:C^\infty(X,\Sigma_gX\otimes F)\longrightarrow C^\infty(X,\Sigma_gX\otimes F)
\]
be the induced action
\[
(U_\Phi\varphi)(x)
:=
\bigl(\gamma_\Sigma\otimes\Phi\bigr)_{\gamma^{-1}x}
\bigl(\varphi(\gamma^{-1}x)\bigr),
\]
where \(\gamma_\Sigma\) denotes the lifted action on the spinor factor.
For \(k\ge 0\), we denote by \(U_\Phi^{(k)}\) the induced action on
\[
C^\infty(X,\Lambda^kT^*X\otimes\Sigma_gX\otimes F),
\]
namely
\[
\bigl(U_\Phi^{(k)}A\bigr)_x(Y_1,\ldots,Y_k)
=
\bigl(\gamma_\Sigma\otimes\Phi\bigr)_{\gamma^{-1}x}
\Bigl(
A_{\gamma^{-1}x}
(d\gamma^{-1}_xY_1,\ldots,d\gamma^{-1}_xY_k)
\Bigr).
\]
Thus \(U_\Phi^{(0)}=U_\Phi\). 
If, for each \(\gamma\in\Gamma\), one chooses a lift \(\Phi_{\gamma,s}\in\Gamma_s\), we also write
\[
U_{\gamma,s}:=U_{\Phi_{\gamma,s}},
\qquad
U_{\gamma,s}^{(k)}:=U_{\Phi_{\gamma,s}}^{(k)} .
\]
Changing the lift by an element of the central subgroup \(U(1)\) only multiplies these operators by a constant unit complex number. 
Therefore all norm identities and equivariance identities below are independent of this auxiliary choice.

In the next proposition we collect the properties of \(U_\Phi\) needed in this paper.

\begin{proposition}
\label{prop:gromov-equivariance}
Let \(\Phi\in\Gamma_s\), and let \(\gamma=\sigma_s(\Phi)\). 
Then the following properties hold.
\begin{enumerate}[label=\textup{(\roman*)}]
    \item for every \(k\ge 0\),
\[
|U_\Phi^{(k)}A|(x)=|A|(\gamma^{-1}x).
\]
\item the action preserves Clifford multiplication: for every tangent vector \(Y\) on \(X\),
\[
c_g(Y)\,U_\Phi\varphi
=
U_\Phi\bigl(c_g(d\gamma^{-1}Y)\varphi\bigr).
\]
\item the twisted connection is \(\Gamma_s\)-equivariant:
\[
\nabla^s(U_\Phi\varphi)
=
U_\Phi^{(1)}(\nabla^s\varphi).
\]
\item the operators \(D_s^+\) and \(D_s^-\) are \(\Gamma_s\)-equivariant.
\end{enumerate}
\end{proposition}

\begin{proof}
These are the standard equivariance properties of the pulled-back spin structure together with Gromov's connection-preserving lift; see \cite[Section~8]{Bal06}.

The norm identity follows because \(\gamma\) is an isometry, the lifted spin action is unitary, and \(\Phi\) is unitary on the line-bundle
factor. 
The Clifford identity follows from the fact that the spin action covers the differential of the deck transformation. 
The spin connection on \(X\) is pulled back from \(M\), hence is preserved by the lifted deck action. 
The line-bundle connection is preserved by the defining property of \(\Gamma_s\). 
Taking the tensor product gives
\[
\nabla^s(U_\Phi\varphi)
=
U_\Phi^{(1)}(\nabla^s\varphi).
\]
Contracting this identity with Clifford multiplication and using the Clifford identity gives
\[
D_sU_\Phi=U_\Phi D_s .
\]
Finally, when \(n\) is even, deck transformations preserve the orientation induced from the spin manifold \(M\). 
Hence the lifted spin action preserves the complex volume element and therefore the chirality splitting.
The chiral equivariance follows.
\end{proof}

By \Cref{prop:gromov-equivariance}, the setting of
\Cref{subsec:equivariant-L2-index} applies with \(G=\Gamma_s\) and
\(K=U(1)\), via the extension
\[
        1\longrightarrow U(1)\longrightarrow\Gamma_s
        \longrightarrow\Gamma\longrightarrow1.
\]
We denote the resulting index by \(\ind^{(2)}_{\Gamma_s}(D_s^+)\).

\begin{proposition}
\label{prop:Gamma-s-index-and-spinors}
For every \(s\in\R\),
\begin{equation}\label{eq:index-formula}
\ind^{(2)}_{\Gamma_s}(D_s^+)
=
\sum_{2\ell+4r=2m}
\frac{(-s)^\ell}{(2\pi)^\ell \ell!}
\int_M
\Ahat_r(\nabla^{TM})\wedge \omega^\ell .
\end{equation}
Consequently, if \(\omega\) is homologically \(\Ahat\)-non-singular and \(\pi^*\omega=d\eta\), then \(X\) admits almost harmonic twisted spinors. More precisely, there exist \(s_j\in(0,1]\), with \(s_j\to0\), and twisted spinors
\[
\psi_j\in
C^\infty(X,\Sigma_gX\otimes F)
\cap
L^2(X,\Sigma_gX\otimes F)
\]
such that
\[
        \|\psi_j\|_{L^2}=1,
        \qquad
        D_{s_j}\psi_j=0 .
\]
\end{proposition}

\begin{proof}
Since \(g=\pi^*g_M\) and
\(F(\nabla^{F,s})=is\,\pi^*\omega\), the canonical local index form
upstairs is the pullback of
\[
        \alpha_s
        :=
        \left[
        \Ahat(\nabla^{TM})
        \wedge
        \exp\left(-\frac{s\omega}{2\pi}\right)
        \right]_{2m}.
\]
Thus \Cref{thm:compact-extension-L2-index} gives
\[
        \ind^{(2)}_{\Gamma_s}(D_s^+)
        =
        \int_M\alpha_s,
\]
and expanding the exponential proves \eqref{eq:index-formula}.

Now set
\[
        I(s):=\ind^{(2)}_{\Gamma_s}(D_s^+).
\]
This is a polynomial in \(s\). Since \(\omega\) is homologically
\(\Ahat\)-non-singular, for some \(p\ge1\) and \(q\ge0\), with
\(2p+4q=2m\),
\[
        \int_M\Ahat_q(\nabla^{TM})\wedge\omega^p\ne0.
\]
The coefficient of \(s^p\) in \(I(s)\) is therefore nonzero. Choose
\(s_j\in(0,1]\), with \(s_j\to0\), such that \(I(s_j)\ne0\) for every
\(j\). By \Cref{cor:nonzero-index-L2-kernel}, \(D_{s_j}\) has a nonzero
smooth \(L^2\)-kernel. Choosing and normalizing an element of each kernel
gives the asserted sequence, which is almost harmonic in the sense of
\Cref{def:almost-harmonic-twisted-spinors}.
\end{proof}

\subsection{Examples and applications}
\label{subsec:examples-applications}

We conclude this section by recording several concrete applications of
\Cref{thm:main-two-form,thm:main-zero-in-the-spectrum,cor:compact-enlargeable-rigidity}.
The first three examples exhibit distinct
ways in which the inequality in \Cref{thm:main-two-form} can be strict. The
last two describe familiar classes of manifolds to which
\Cref{cor:compact-enlargeable-rigidity} applies.

\begin{example}\label{ex:product-surfaces-strict}
Let \(M=\Sigma_{g_1}\times\cdots\times\Sigma_{g_m}\), where each
\(\Sigma_{g_i}\) is a closed oriented surface of genus \(g_i\geq 2\), and
\(m\geq 2\). Then \(M\) is spin and aspherical. If \(\omega_i\) is an area
form on \(\Sigma_{g_i}\), then
\[
        \omega=\sum_i\operatorname{pr}_i^*\omega_i
\]
is homologically non-singular and the lift of \(\omega\) to the universal
cover is exact. Hence \Cref{thm:main-two-form} applies to every Riemannian
metric on \(M\).

Moreover, the inequality is strict for every such metric: by Brooks' theorem
\cite{Brooks1981}, \(\lambda_0>0\), since \(\pi_1(M)\) is non-amenable. If
equality held, \Cref{thm:main-two-form} would force
\((\widetilde M,\widetilde g)\) to be real hyperbolic. This is impossible by
Preissmann's theorem (see, for instance,
\cite[Chapter~12, Theorem~3.2]{doCarmoRiemannianGeometry}), since
\(\pi_1(M)\) contains \(\mathbb Z^2\).
\end{example}

\begin{example}\label{ex:symmetric-product-strict}
There are also strict examples with \(\lambda_0=0\). For \(n\geq 2\) and
\(h\geq 2n-1\) such that \(n-h\) is odd, let
\[
        M_h:=\operatorname{SP}^n(\Sigma_h)
\]
be the symmetric product considered by Di Cerbo--Dranishnikov--Jauhari
\cite{DiCerboDranishnikovJauhari2025}. By
\cite[Theorems~4.1 and~9.12]{DiCerboDranishnikovJauhari2025}, \(M_h\) is a
closed virtually spin manifold carrying a symplectically aspherical K\"ahler
form. In particular, the form is homologically non-singular and its lift to
the universal cover is exact. Moreover,
\cite[Corollary~3.5 and Propositions~3.6--3.7]{DiCerboDranishnikovJauhari2025}
identifies \(\pi_1(M_h)\) with \(\mathbb Z^{2h}\) and shows that \(\pi_2(M_h)\)
is infinite. Hence Brooks' theorem \cite{Brooks1981} gives
\[
        \lambda_0(\widetilde M_h,\widetilde g_h)=0
\]
for every Riemannian metric \(g_h\) on \(M_h\). Since \(\widetilde M_h\) is
not contractible, it cannot be Euclidean. Therefore the equality case in
\Cref{thm:main-two-form} is impossible, and
\[
        \inf_{M_h}\scal_{g_h}<0
\]
for every Riemannian metric \(g_h\) on \(M_h\). The nonexistence of metrics
of positive scalar curvature on \(M_h\) was previously established in
\cite[Theorem~10.9(1)]{DiCerboDranishnikovJauhari2025}.
\end{example}

\begin{example}\label{ex:ahat-refined-strict}
Finally, strictness can be detected by the \(\Ahat\)-refinement even when the
top-power condition does not apply. Let \(p\geq 1\) and \(q\geq 0\). If
\(Y^{4q}\) is closed, connected, and spin with
\(\int_Y\Ahat_q(TY)\neq 0\), and \(B^{2p}\) is closed, connected, and spin and
carries a closed two-form \(\omega_B\) with exact lift and
\(\int_B\omega_B^p\neq 0\), then on \(M=Y\times B\), the form
\(\omega=\operatorname{pr}_B^*\omega_B\) satisfies
\[
        \int_M\Ahat_q(TM)\wedge\omega^p\neq 0.
\]
If \(q>0\) and \(Y\) is simply connected, \(M\) does not carry any closed
two-form \(\omega\) whose top power is nonzero and whose lift to the
universal cover is exact. For instance, take \(Y\) to be a \(K3\) surface
and \(B=\Sigma_{g_1}\times\cdots\times\Sigma_{g_p}\), with \(g_i\geq 2\).
In this case \(\pi_1(M)=\pi_1(B)\) is non-amenable, so Brooks' theorem
\cite{Brooks1981} gives \(\lambda_0(\widetilde M,\widetilde g)>0\) for every
Riemannian metric \(g\) on \(M\). Equality would force
\((\widetilde M,\widetilde g)\) to be real hyperbolic, which is impossible
since \(\widetilde M\cong Y\times\widetilde B\) is not contractible. Thus the
inequality is strict for every Riemannian metric on \(M\).
\end{example}

\begin{example}\label{ex:closed-hyperbolic}
Let \(M^n\), \(n\geq 3\), be a closed hyperbolic manifold. Its fundamental
group is residually finite, so \(M\) is compactly enlargeable by
\cite[Proposition~3.3]{GromovLawson1980}. Moreover,
\(M\) is virtually spin by \cite[Theorem~4.1]{LongReidVirtuallySpinning},
after first passing to the orientation double cover if necessary. Suppose
\(g_M\) is any Riemannian metric on \(M\). Let
\(\pi:(X,g)\to(M,g_M)\) be the universal Riemannian cover and set
\(\lambda_0:=\lambda_0(X,g)\). By \cite{Brooks1981}, \(\lambda_0>0\). By
\Cref{cor:compact-enlargeable-rigidity},
\[
    \inf_M \scal_{g_M}
    \leq
    -\frac{4n}{n-1}\lambda_0,
\]
and equality holds if and only if \(g_M\) is hyperbolic with sectional
curvature
\[
    \sec_{g_M}
    \equiv
    -\frac{4\lambda_0}{(n-1)^2}.
\]
\end{example}

\begin{example}\label{ex:three-dimensional}
Let \(M^3\) be a closed connected oriented three-manifold which does not
admit a metric of positive scalar curvature, and let \(g_M\) be any
Riemannian metric on \(M\). Let \(\pi:(X,g)\to(M,g_M)\) be its universal
Riemannian cover and set \(\lambda_0:=\lambda_0(X,g)\). The classification
of closed oriented three-manifolds carrying metrics of positive scalar
curvature \cite[Section~1]{Marques2012}, together with prime decomposition
and geometrization
\cite[Theorem~1.7.3 and Statement~(C.1) in Section~3.2]{AschenbrennerFriedlWilton},
shows that the prime decomposition of \(M\) contains an aspherical summand
\(N\). Every oriented three-manifold is
spin, and \(\pi_1(N)\) is residually finite
\cite[Statements~(C.28)--(C.29)]{AschenbrennerFriedlWilton}. If \(\pi_1(N)\)
contains \(\mathbb Z^2\), then \cite[Theorem~6.1]{GromovLawson1980}, applied
to the expandable torus \(T^2\) and a map \(T^2\to N\) inducing this inclusion,
shows that \(N\) is compactly enlargeable. If \(\pi_1(N)\) contains no
\(\mathbb Z^2\), then \(N\) is atoroidal in the sense of
\cite[Section~1.6]{AschenbrennerFriedlWilton}. Since \(N\) is irreducible and
has infinite fundamental group, the Hyperbolization Theorem
\cite[Theorem~1.7.5]{AschenbrennerFriedlWilton} shows that \(N\) is hyperbolic.
Together with the residual finiteness established above,
\cite[Proposition~3.3]{GromovLawson1980} then shows that \(N\) is compactly
enlargeable. Compact enlargeability
is stable under connected sums with spin manifolds
\cite[p.~210]{GromovLawson1980}, so \(M\) is compactly enlargeable. Therefore
\Cref{cor:compact-enlargeable-rigidity} gives
\[
    \inf_M\scal_{g_M}
    \leq -6\lambda_0.
\]
Equality holds if and only if \(g_M\) has constant sectional curvature
\[
    \sec_{g_M}\equiv-\lambda_0.
\]
Thus an equality metric is flat when \(\lambda_0=0\) and real hyperbolic
when \(\lambda_0>0\).
\end{example}

\section{Scalar upper bound and scalar rigidity}
\label{sec:scalar-rigidity}

In this section we prove the sharp scalar comparison and its equality case from the almost harmonic twisted spinors constructed in the preceding section.

Throughout this section, \((M^n,g_M)\), \(n\ge3\), is a closed connected spin Riemannian manifold, and \(\pi:(X,g)\longrightarrow(M,g_M)\) is its universal Riemannian covering. 
We write \(\lambda_0=\lambda_0(X,g)\) for the bottom of the \(L^2\)-spectrum of \(\Delta_g=d^*d\). 
Let \(\omega\in\Omega^2(M)\) be a closed two-form whose lift to \(X\) is exact,
\[
        \pi^*\omega=d\eta.
\]
As in \Cref{subsec:twisted-dirac}, \(F=X\times\C\) carries the family of unitary connections \(\nabla^{F,s}=d+is\eta\), \(s\in\R\), and \(D_s\) denotes the corresponding twisted Dirac operator on \(\Sigma_gX\otimes F\). 
We denote the associated twisted connection by
\[
        \nabla^s=\nabla^{\Sigma_gX}\otimes1+1\otimes\nabla^{F,s}.
\]

\begin{theorem}
\label{thm:scalar-upper-bound}
Assume that there exist \(s_j\in[0,1]\), with \(s_j\to0\), and twisted spinors
\[
        \psi_j\in C^\infty(X,\Sigma_gX\otimes F)
        \cap L^2(X,\Sigma_gX\otimes F)
\]
such that, after normalization,
\[
        \|\psi_j\|_{L^2}=1,
        \qquad
        \|D_{s_j}\psi_j\|_{L^2}\longrightarrow0.
\]
Then
\begin{equation}
\label{eq:scalar-upper-bound}
        \inf_M\scal_{g_M}
        \le
        -\frac{4n}{n-1}\lambda_0(X,g).
\end{equation}
Moreover, if equality holds in \eqref{eq:scalar-upper-bound}, then
\begin{equation}
\label{eq:scalar-rigidity}
        \scal_{g_M}\equiv -\frac{4n}{n-1}\lambda_0(X,g).
\end{equation}
\end{theorem}

\subsection{The Kato defect}
\label{subsec:refined-kato-defect}

Let \(E\to X\) be a Hermitian vector bundle with unitary connection, and let
\[
        D^E:C^\infty(X,\Sigma_gX\otimes E)
        \longrightarrow
        C^\infty(X,\Sigma_gX\otimes E)
\]
be the associated twisted Dirac operator. 
Let \(q:X\to[0,\infty)\) be a continuous function and set
\[
        U_q:=\{x\in X:q(x)>0\}.
\]
Assume that \(q|_{U_q}\) is smooth. 
For \(\psi\in C^\infty(X,\Sigma_gX\otimes E)\), define the \(q\)-Kato defect on \(U_q\) by
\begin{equation}
\label{eq:kato-defect-definition}
        P_q^E(\psi)(Y)
        :=
        \nabla_Y^{\Sigma_gX\otimes E}\psi
        -\frac n{n-1}Y(\log q)\psi
        -\frac1{n-1}c(Y)c(\nabla^g\log q)\psi.
\end{equation}
When \(q>0\) is smooth, this is defined globally. 
When \(q=r:=|\psi|\), we extend \(P_r^E(\psi)\) by zero on \(X\setminus U_r\). 
This convention is compatible with the weak interpretation, since \(\nabla\psi=0\) almost everywhere on the zero set of \(\psi\); see \cite[Lemma~7.7]{GilbargTrudinger} applied to local components.

When \(E=F\) is equipped with \(\nabla^{F,s}=d+is\eta\), we write
\[
        P_q^s:=P_q^F.
\]

\begin{lemma}
\label{lem:kato-defect-L2}
Let \(\psi\in C^\infty(X,\Sigma_gX\otimes E)\), and set \(r:=|\psi|\).
Then \(P^E_r(\psi)\) defines an element of \(L^2_{\mathrm{loc}}\bigl(X,T^*X\otimes\Sigma_gX\otimes E\bigr)\).
Moreover, if \(\nabla^{\Sigma_gX\otimes E}\psi\in L^2\bigl(X,T^*X\otimes\Sigma_gX\otimes E\bigr)\),
then \(P^E_r(\psi)\in L^2\bigl(X,T^*X\otimes\Sigma_gX\otimes E\bigr)\).
\end{lemma}

\begin{proof}
The function \(r=|\psi|\) is locally Lipschitz. On \(U_r\), for every vector field \(Y\),
\[
        dr(Y)=\frac{\operatorname{Re}\langle
        \nabla_Y^{\Sigma_gX\otimes E}\psi,\psi\rangle}{|\psi|}.
\]
Moreover, \(dr=0\) almost everywhere on \(X\setminus U_r\). Hence
\[
        |dr|\le |\nabla^{\Sigma_gX\otimes E}\psi|
        \qquad\text{a.e.}
\]
On \(U_r\), the singular terms in \(P_r^E(\psi)\) are controlled by \(dr\):
\[
        |Y(\log r)\psi|=|Y(r)|,
        \qquad
        |c(Y)c(\nabla\log r)\psi|\le |Y|\,|dr|.
\]
Therefore, for a dimensional constant \(C_n>0\),
\[
        |P_r^E(\psi)|
        \le
        C_n\bigl(|\nabla^{\Sigma_gX\otimes E}\psi|+|dr|\bigr)
        \le
        2C_n|\nabla^{\Sigma_gX\otimes E}\psi|
\]
almost everywhere on \(U_r\). On \(X\setminus U_r\), we have extended \(P_r^E(\psi)\) by zero, so the same bound holds trivially. The local and global \(L^2\)-claims follow.
\end{proof}

We shall use the following pointwise identity, which isolates the equality defect in the refined Kato inequality for Dirac operators. 
In the harmonic case, the Dirac term vanishes and the nonnegative quantity \(|P^E_{|\psi|}(\psi)|^2\) is precisely the gap in the sharp refined Kato inequality; see, for example, \cite[Theorem~3.1]{CalderbankGauduchonHerzlich}.
For the Kato inequality in the non-harmonic case, see \cite[Section~4.3, proof of Lemma~4.2]{Davaux2003}, and \cite[Section~3]{WangZhu2024}.

\begin{proposition}
\label{prop:refined-kato-identity}
Let \((X^n,g)\), \(n\ge2\), be a Riemannian spin manifold. 
Let \(E\to X\) be a Hermitian bundle with unitary connection. 
Let \(\psi\in C^\infty(X,\Sigma_gX\otimes E)\) and set \(r:=|\psi|\).
Then, pointwise on \(U_r\),
\begin{equation}
\label{eq:refined-kato-identity}
        |P_r^E(\psi)|^2
        =
        |\nabla^{\Sigma_gX\otimes E}\psi|^2
        -\frac n{n-1}|dr|^2
        +\frac2{n-1}\operatorname{Re}
        \langle D^E\psi,c(\nabla^g\log r)\psi\rangle.
\end{equation}
\end{proposition}

\begin{proof}
The computation is pointwise. 
Fix \(x\in U_r\), and choose a local \(g\)-orthonormal frame \(e_1,\ldots,e_n\). 
All expressions below are evaluated at \(x\). Put
\[
        u:=\log r,
        \qquad
        V:=\nabla^g u,
        \qquad
        a_i:=e_i(u),
        \qquad
        A_i:=\nabla_{e_i}^{\Sigma_gX\otimes E}\psi.
\]
Define
\[
        B_i:=\frac n{n-1}a_i\psi
        +\frac1{n-1}c(e_i)c(V)\psi.
\]
Then \(P_r^E(\psi)(e_i)=A_i-B_i\). Using
\[
        c(e_i)c(V)+c(V)c(e_i)=-2a_i
\]
and the skew-Hermiticity of Clifford multiplication, one obtains
\[
        \sum_i |B_i|^2=\frac n{n-1}|dr|^2.
\]
Moreover,
\[
        \operatorname{Re}\langle A_i,\psi\rangle=r\,e_i(r)=r^2a_i,
\]
and
\[
        \sum_i\langle A_i,c(e_i)c(V)\psi\rangle
        =
        -\langle D^E\psi,c(V)\psi\rangle.
\]
Hence
\[
        \operatorname{Re}\sum_i\langle A_i,B_i\rangle
        =
        \frac n{n-1}|dr|^2
        -\frac1{n-1}\operatorname{Re}\langle D^E\psi,c(\nabla^g\log r)\psi\rangle.
\]
Substituting this into
\[
        |P_r^E(\psi)|^2=\sum_i|A_i-B_i|^2
\]
gives \eqref{eq:refined-kato-identity}.
\end{proof}

\begin{corollary}
\label{cor:harmonic-kato-identity}
Let \(D^E\psi=0\), set \(r=|\psi|\), and assume \(\psi,\nabla^{\Sigma_gX\otimes E}\psi\in L^2\).
Then
\begin{equation}
\label{eq:integrated-harmonic-kato}
        \int_X |P_r^E(\psi)|^2\dmu
        =
        \|\nabla^{\Sigma_gX\otimes E}\psi\|_{L^2}^2
        -\frac n{n-1}\int_X |dr|^2\dmu.
\end{equation}
In particular,
\begin{equation}
\label{eq:refined-kato-inequality-harmonic}
        \frac n{n-1}\int_X |dr|^2\dmu
        \le
        \|\nabla^{\Sigma_gX\otimes E}\psi\|_{L^2}^2.
\end{equation}
\end{corollary}

\begin{proof}
The pointwise identity \eqref{eq:refined-kato-identity} holds on \(U_r\). Since \(P_r^E(\psi)\) is extended by zero on \(X\setminus U_r\), and since \(\nabla\psi=0\) and \(dr=0\) almost everywhere on \(X\setminus U_r\), the identity holds almost everywhere on \(X\). The term involving \(D^E\psi\) vanishes. Integrating gives \eqref{eq:integrated-harmonic-kato}. The inequality follows because the left-hand side of \eqref{eq:integrated-harmonic-kato} is nonnegative.
\end{proof}

For almost harmonic spinors we shall use the following immediate consequence.

\begin{corollary}
\label{cor:almost-harmonic-kato-estimate}
Let \(\psi\in C^\infty(X,\Sigma_gX\otimes E)\) and set \(r=|\psi|\).
Assume that \(\psi,D^E\psi\in L^2\) and that
\(\nabla^{\Sigma_gX\otimes E}\psi\in L^2\). Then
\begin{equation}
\label{eq:almost-harmonic-kato-estimate}
        \|\nabla^{\Sigma_gX\otimes E}\psi\|_{L^2}^2
        \ge
        \frac n{n-1}\int_X|dr|^2\dmu
        -\frac2{n-1}\|D^E\psi\|_{L^2}\|dr\|_{L^2}.
\end{equation}
\end{corollary}

\begin{proof}
Integrate \eqref{eq:refined-kato-identity} over \(U_r\), use the convention on the zero set of \(r\), and observe that
\[
        |c(\nabla^g\log r)\psi|=|dr|
        \qquad\text{on }U_r.
\]
The conclusion follows from the nonnegativity of \(|P_r^E(\psi)|^2\) and the Cauchy--Schwarz inequality.
\end{proof}

\subsection{Sharp scalar upper bound}
\label{subsec:sharp-scalar-upper-bound}

\begin{proposition}
\label{prop:scalar-upper-bound}
Under the hypotheses of \Cref{thm:scalar-upper-bound}, inequality \eqref{eq:scalar-upper-bound} holds.
\end{proposition}

\begin{proof}
Put \(\tau_j:=\|D_{s_j}\psi_j\|_{L^2}\).  Thus \(\tau_j\to0\).
Put
\[
        \sigma_0:=\inf_M\scal_{g_M}=\inf_X\scal_g,
        \qquad
        r_j:=|\psi_j|,
        \qquad
        \alpha_n:=\frac n{n-1}.
\]
The integrated Lichnerowicz formula first gives a uniform bound for
\(\|\nabla^{s_j}\psi_j\|_{L^2}\), and hence for \(\|dr_j\|_{L^2}\).
Therefore, for some positive constant $C>0$,  \Cref{cor:almost-harmonic-kato-estimate} gives
\[
        \alpha_n\int_X|dr_j|^2\dmu
        \leq \|\nabla^{s_j}\psi_j\|_{L^2}^2+C\tau_j .
\]
Together with \(\int_Xr_j^2\dmu=1\), the integrated twisted
Lichnerowicz formula, \eqref{eq:Romega-bound}, and Rayleigh's inequality
give
\[
\begin{aligned}
        \tau_j^2
        =&
        \|\nabla^{s_j}\psi_j\|_{L^2}^2
        +\frac14\int_X\scal_g r_j^2\dmu
        +s_j\int_X\langle R_\omega\psi_j,\psi_j\rangle\dmu\\
        \ge&\alpha_n\int_X|dr_j|^2\dmu+\frac{\sigma_0}{4}
        - C_\omega s_j-C\tau_j\\
        \ge&\alpha_n\lambda_0+\frac{\sigma_0}{4}
        -C_\omega s_j-C\tau_j.
\end{aligned}
\]
Letting \(j\to\infty\), we get \(\sigma_0/4\le-\alpha_n\lambda_0\). 
Thus
\[
        \inf_M\scal_{g_M}=\sigma_0
        \le
        -4\alpha_n\lambda_0
        =
        -\frac{4n}{n-1}\lambda_0(X,g).\qedhere
\]
\end{proof}

\subsection{Generalized ground states}
\label{subsec:generalized-ground-states}

\begin{definition}
\label{def:generalized-ground-state}
A generalized ground state for \((X,g)\) is a positive smooth function \(f\in C^\infty(X)\) satisfying
\[
        \Delta_g f=\lambda_0 f.
\]
\end{definition}

\begin{proposition}
\label{prop:ground-state-existence}
Every complete connected Riemannian manifold admits a generalized ground state.
\end{proposition}

\begin{proof}
If \(X\) is compact, then \(\lambda_0=0\) and the conclusion holds with
\(f\equiv1\); hence assume that \(X\) is noncompact.
By Rayleigh's inequality \eqref{eq:rayleigh-inequality},
\[
        \int_X(|du|^2-\lambda_0u^2)\dmu\ge0,
        \qquad
        u\in C_c^\infty(X).
\]
The theorem of Fischer-Colbrie--Schoen \cite[Theorem~1]{FischerColbrieSchoen1980} yields a smooth positive solution of \(\Delta_g f=\lambda_0f\).
\end{proof}

We now study the interaction of \(f\) with functions in \(W^{1,2}(X)\).
The identity below is the classical ground-state transform (also called the
Jacobi identity); see
\cite[Theorem~2.3 and Corollary~2.9]{LenzStollmannVeselic2009} for a general
strongly local Dirichlet-form formulation. We include the proof in order to
record explicitly the extension from compactly supported smooth functions to
all of \(W^{1,2}(X)\).

\begin{lemma}
\label{lem:ground-state-transform}
Let \(f>0\) be a smooth function satisfying \(\Delta_g f=\lambda_0f\). Then,
for every \(v\in W^{1,2}(X)\), the function \(v/f\) belongs to
\(W^{1,2}_\loc(X)\), the weighted differential
\(f\,d(v/f)\) belongs to \(L^2(X,T^*X)\), and
\begin{equation}
\label{eq:ground-state-transform}
        \int_X (|dv|^2-\lambda_0v^2)\,d\mu_g
        =
        \int_X f^2\left|d\left(\frac vf\right)\right|^2\,d\mu_g.
\end{equation}
\end{lemma}

\begin{proof}
First assume \(v\in C_c^\infty(X)\), and put \(h=v/f\). Then
\[
        d(fh)=h\,df+f\,dh
\]
and therefore
\[
        |d(fh)|^2=h^2|df|^2+f^2|dh|^2+2fh\langle df,dh\rangle.
\]
Testing \(\Delta_g f=\lambda_0f\) against \(fh^2\) gives
\[
        \lambda_0\int_X f^2h^2\,d\mu_g
        =
        \int_X h^2|df|^2+2fh\langle df,dh\rangle\,d\mu_g.
\]
Combining the previous two identities gives
\begin{equation}\label{eq:ground-transform-1}
\int_X (|dv|^2-\lambda_0v^2)\,d\mu_g
=\int_X f^2|dh|^2\,d\mu_g
=\int_X f^2\left|d\left(\frac vf\right)\right|^2\,d\mu_g.
\end{equation}

Now let \(v\in W^{1,2}(X)\). Since \((X,g)\) is complete, choose \(v_k\in C_c^\infty(X)\) with \(v_k\to v\) in \(W^{1,2}(X)\). 
Put \(h_k:=v_k/f\). 
Applying the identity \eqref{eq:ground-transform-1} to \(v_k-v_\ell=f(h_k-h_\ell)\), we get
\begin{equation}\label{eq:ground-transform-2}
\int_X f^2|d(h_k-h_\ell)|^2\,d\mu_g
=\int_X\bigl(|d(v_k-v_\ell)|^2-\lambda_0(v_k-v_\ell)^2\bigr)\,d\mu_g.
\end{equation}
The right-hand side tends to zero as \((v_k)\) is Cauchy in \(W^{1,2}(X)\). 
Hence \((f\,dh_k)\) is Cauchy in \(L^2(X,T^*X)\). Let \(\beta\) be its \(L^2\)-limit.

Let \(\Omega\subset X\) be an open relatively compact set.
Then \(f^{-1}\) is bounded on \(\Omega\). Thus
\[
        h_k\to \frac vf \quad\text{in }L^2(\Omega),
        \qquad
        dh_k\to f^{-1}\beta \quad\text{in }L^2(\Omega,T^*X).
\]
Since \(\Omega\) was arbitrary, we conclude that \(v/f\in W^{1,2}_\loc(X)\) and \(d(v/f)=f^{-1}\beta\) almost everywhere on \(X\).
In particular, \(f\,d\left(v/f\right)\in L^2(X,T^*X)\).
Applying the identity \eqref{eq:ground-transform-1} to \(v_k\), we get
\[
\int_X\bigl(|dv_k|^2-\lambda_0v_k^2\bigr)\,d\mu_g
=\int_X f^2|dh_k|^2\,d\mu_g
=\int_X |f\,dh_k|^2\,d\mu_g .
\]
Since \(v_k\to v\) in \(W^{1,2}(X)\) and \(f\,dh_k\to\beta=f\,d\left(v/f\right)\) in \(L^2(X,T^*X)\), passing to the limit in the previous identity yields \eqref{eq:ground-state-transform}.
\end{proof}

\subsection{A Wang--Zhu-type positivity principle}
\label{subsec:positivity-principle}

The following proposition is the form of the Wang--Zhu positivity principle needed for scalar rigidity. 
Wang--Zhu prove a considerably more general quantitative estimate for low-energy sections on normal covers \cite[Theorem~5.1]{WangZhu2024}. 
In the specific case needed here, we instead give a direct proof using the
ground-state transform from \Cref{lem:ground-state-transform} and the
Poincar\'e inequality.

\begin{proposition}
\label{prop:wang-zhu-positivity}
Let \(a>0\), let \(V\in C^\infty(M,\R)\), and write \(\widetilde V:=\pi^*V\). Assume that
\[
        a\lambda_0(X,g)+V\ge0\quad\text{on }M,
\]
and that \(a\lambda_0(X,g)+V\) is positive at some point of \(M\). Then there exists \(\kappa>0\) such that
\begin{equation}
\label{eq:wang-zhu-positivity}
        a\int_X |du|^2\dmu
        +
        \int_X \widetilde V |u|^2\dmu
        \ge
        \kappa\int_X |u|^2\dmu
\end{equation}
for every \(u\in W^{1,2}(X)\).
\end{proposition}

\begin{lemma}
\label{lem:ground-state-poincare}
Let \(U_0\subset M\) be a nonempty open set and set \(U:=\pi^{-1}(U_0)\). Let \(f>0\) be a generalized ground state. Then there exists a constant \(C_U>0\) such that, for every \(u\in C_c^\infty(X)\),
\begin{equation}
\label{eq:ground-state-poincare}
        \int_X |u|^2\dmu
        \le
        C_U\left(
        \int_X f^2\left|d\left(\frac uf\right)\right|^2\dmu
        +
        \int_U |u|^2\dmu
        \right).
\end{equation}
\end{lemma}

\begin{proof}
Let \(\Gamma\) be the deck group. Choose \(x_*\in U\). Since \(M\) is compact, there exists a relatively compact connected open set \(\Omega\Subset X\), with smooth boundary, such that
\begin{equation*}
        x_*\in\Omega,
        \qquad
        X=\bigcup_{\gamma\in\Gamma}\gamma\Omega.
\end{equation*}
Set \(O:=\Omega\cap U\). Then \(O\) is a nonempty open subset of \(\Omega\). By the Poincar\'e inequality with control on a nonempty open subset, there exists \(C_P>0\) such that
\begin{equation}\label{eq:Poincare'}
        \int_\Omega |\phi|^2\dmu
        \le
        C_P\left(\int_\Omega |d\phi|^2\dmu+
        \int_O |\phi|^2\dmu\right)
\end{equation}
for all \(\phi\in W^{1,2}(\Omega)\). By deck invariance, the same estimate holds on each translate \(\gamma\Omega\), with the same constant.

Since \(f\) is a positive solution of \((\Delta_g-\lambda_0)f=0\), the local Harnack inequality, applied on a finite cover of \(\Omega\) and translated by deck transformations, gives a constant \(C_H\ge1\) such that
\begin{equation}\label{eq:Harnack}
        \sup_{\gamma\Omega}f
        \le
        C_H\inf_{\gamma\Omega}f
        \qquad\text{for every }\gamma\in\Gamma.
\end{equation}
Let \(u\in C_c^\infty(X)\), and put \(h=u/f\). Applying the inequality \eqref{eq:Poincare'} to \(h\) on \(\gamma\Omega\), multiplying by \(\sup_{\gamma\Omega}f^2\), and using \eqref{eq:Harnack}, we obtain
\begin{equation}\label{eq:Poincare'+Harnack}
        \int_{\gamma\Omega}|u|^2\dmu
        \le
        C_PC_H^2\left(
        \int_{\gamma\Omega}f^2|dh|^2\dmu+
        \int_{\gamma O}|u|^2\dmu
        \right).
\end{equation}
The family \(\{\gamma\Omega\}_{\gamma\in\Gamma}\) has uniformly bounded overlap. 
Hence there exists \(N_\Omega<\infty\) such that, for every nonnegative measurable function \(G\),
\[
        \sum_{\gamma\in\Gamma}\int_{\gamma\Omega}G\,d\mu_g
        \leq
        N_\Omega\int_X G\,d\mu_g .
\]
The same estimate holds with \(\Omega\) replaced by \(O\subset \Omega\).
Summing the inequality \eqref{eq:Poincare'+Harnack} over \(\gamma\in\Gamma\), and using \(\gamma O\subset U\) gives \eqref{eq:ground-state-poincare} with \(C_U=C_PC_H^2N_\Omega\).
\end{proof}

\begin{proof}[Proof of \Cref{prop:wang-zhu-positivity}]
Set
\[
        W:=a\lambda_0+V,
        \qquad
        \widetilde W:=\pi^*W=a\lambda_0+\widetilde V.
\]
By assumption, \(W\ge0\) on \(M\), and \(W\) is positive somewhere. Hence there exist a nonempty open set \(U_0\subset M\) and a number \(\delta>0\) such that \(W\ge\delta\) on \(U_0\). Let \(U:=\pi^{-1}(U_0)\), and choose a generalized ground state \(f>0\).

For \(u\in C_c^\infty(X)\), combining \Cref{lem:ground-state-transform} with \Cref{lem:ground-state-poincare}, we get
\[
        \int_X|u|^2\dmu
        \le
        C_U\left(
        \int_X(|du|^2-\lambda_0u^2)\dmu+
        \int_U |u|^2\dmu
        \right).
\]
Therefore
\begin{align*}
        a\int_X|du|^2\dmu+
        \int_X\widetilde V|u|^2\dmu
        &=
        a\int_X(|du|^2-\lambda_0u^2)\dmu+
        \int_X\widetilde W|u|^2\dmu \\
        &\ge
        a\int_X(|du|^2-\lambda_0u^2)\dmu+
        \delta\int_U|u|^2\dmu \\
        &\ge
        \min\{a,\delta\}
        \left(
        \int_X(|du|^2-\lambda_0u^2)\dmu+
        \int_U|u|^2\dmu
        \right) \\
        &\ge
        \frac{\min\{a,\delta\}}{C_U}
        \int_X|u|^2\dmu.
\end{align*}
This proves \eqref{eq:wang-zhu-positivity} for compactly supported smooth \(u\). Since \(X\) is complete, \(C_c^\infty(X)\) is dense in \(W^{1,2}(X)\), and since \(\widetilde V\in L^\infty(X)\), the estimate extends to all \(u\in W^{1,2}(X)\).
\end{proof}

\subsection{Scalar rigidity}
\label{subsec:scalar-rigidity-proof}

\begin{proposition}
\label{prop:scalar-rigidity}
Assume the hypotheses of \Cref{thm:scalar-upper-bound}. Suppose equality holds in \eqref{eq:scalar-upper-bound}. Then
\[
        \scal_{g_M}\equiv -\frac{4n}{n-1}\lambda_0(X,g).
\]
\end{proposition}

\begin{proof}
After normalizing the spinors, set \(r_j:=|\psi_j|\) and \(\alpha_n:=n/(n-1)\). By equality in \eqref{eq:scalar-upper-bound},
\[
        \inf_M\scal_{g_M}=-4\alpha_n\lambda_0.
\]
Suppose, by contradiction, that \(\scal_{g_M}\) is not constant. Then
\[
        \frac14\scal_{g_M}+\alpha_n\lambda_0\ge0
\]
on \(M\), and this function is positive at some point. Applying \Cref{prop:wang-zhu-positivity} with \(a=\alpha_n\) and \(V=\frac14\scal_{g_M}\), we obtain \(c_0>0\) such that, for every \(u\in W^{1,2}(X)\),
\begin{equation*}
        \alpha_n\int_X|du|^2\dmu
        +\frac14\int_X\scal_g u^2\dmu
        \ge
        c_0\int_Xu^2\dmu.
\end{equation*}
Applying this to \(u=r_j\) and using \(\|r_j\|_{L^2}=1\) gives
\begin{equation}
\label{eq:scalar-rigidity-positive-on-rj}
        \alpha_n\int_X|dr_j|^2\dmu
        +\frac14\int_X\scal_g r_j^2\dmu
        \ge
        c_0.
\end{equation}
Put \(\tau_j:=\|D_{s_j}\psi_j\|_{L^2}\).  As in the proof of
\Cref{prop:scalar-upper-bound}, the covariant derivatives and the differentials
\(dr_j\) are uniformly bounded in \(L^2\), and
\Cref{cor:almost-harmonic-kato-estimate} gives
\[
        \|\nabla^{s_j}\psi_j\|_{L^2}^2
        \geq \alpha_n\int_X|dr_j|^2\dmu-C\tau_j.
\]
Using this estimate, \eqref{eq:integrated-lichnerowicz},
\eqref{eq:Romega-bound}, and \eqref{eq:scalar-rigidity-positive-on-rj}, we obtain
\[
\begin{aligned}
        \tau_j^2&=
        \|\nabla^{s_j}\psi_j\|_{L^2}^2
        +\frac14\int_X\scal_g r_j^2\dmu
        +s_j\int_X\langle R_\omega\psi_j,\psi_j\rangle\dmu\\
        &\geq\alpha_n\int_X|dr_j|^2\dmu
        +\frac14\int_X\scal_g r_j^2\dmu
        -C_\omega s_j-C\tau_j\\
        &\geq c_0-C_\omega s_j-C\tau_j.
\end{aligned}
\]
Letting \(j\to\infty\) gives a contradiction. 
Therefore \(\scal_{g_M}\) is constant, and the equality assumption determines the constant.
\end{proof}

\begin{proof}[Proof of \Cref{thm:scalar-upper-bound}]
The scalar upper bound is \Cref{prop:scalar-upper-bound}. The equality case is \Cref{prop:scalar-rigidity}.
\end{proof}

\section{Geometric rigidity}
\label{sec:einstein-rigidity}

In this section we prove the analytic rigidity statement \Cref{thm:intro-analytic-einstein}, which upgrades scalar rigidity to Einstein rigidity. 
Once the scalar curvature has attained the sharp constant value, the task is to understand the equality case in the Rayleigh--Kato--Lichnerowicz inequalities. 
The key point is that the Kato defect is also the defect in the conformal spinor transformation formula.

We first show that the equality defects give small conformal Kato defect with respect to a fixed generalized ground state. 
We then recenter the harmonic twisted spinors by deck transformations and pass to a conformal limit carrying a nonzero parallel spinor. 
This makes the limiting conformal metric Ricci-flat, and a conformal Einstein criterion then yields the Einstein equation for the original metric.

Throughout this section, \((M^n,g_M)\), \(n\ge3\), is a closed connected spin
Riemannian manifold, and
\[
        \pi:(X,g)\longrightarrow(M,g_M)
\]
is its universal Riemannian covering. Let \(\omega\in\Omega^2(M)\) be a
closed two-form whose lift is exact,
\[
        \pi^*\omega=d\eta
\]
for some real one-form \(\eta\in\Omega^1(X)\). As in
\Cref{subsec:twisted-dirac}, \(F=X\times\C\) is equipped with the unitary
connection \(d+is\eta\), and \(D_s\) denotes the corresponding twisted Dirac
operator on \(\Sigma_gX\otimes F\).

\begin{theorem}
\label{thm:einstein-rigidity}
Assume that
\begin{equation}
\label{eq:sharp-scalar-equality-section5}
        \scal_g\equiv -\frac{4n}{n-1}\lambda_0.
\end{equation}
Assume moreover that there exist \(s_j\in[0,1]\), with \(s_j\to0\), and
twisted spinors
\[
        \psi_j\in C^\infty(X,\Sigma_gX\otimes F)
        \cap L^2(X,\Sigma_gX\otimes F)
\]
such that
\begin{equation}
\label{eq:almost-harmonic-twisted-spinors-section5}
        \|\psi_j\|_{L^2}=1,
        \qquad
        \|D_{s_j}\psi_j\|_{L^2}\longrightarrow0.
\end{equation}
Then
\[
        \Ric_g=-\frac{4\lambda_0}{n-1}g.
\]
Furthermore, if \(\lambda_0>0\), then \((X,g)\) is isometric to real
hyperbolic space with constant sectional curvature
\[
        \sec_g=-\frac{4\lambda_0}{(n-1)^2}.
\]
\end{theorem}

\subsection{A conformal criterion for geometric rigidity}
\label{subsec:conformal-geometric-rigidity-criterion}

We first isolate an abstract conformal criterion that contains both geometric
conclusions of \Cref{thm:intro-analytic-einstein}. Its Einstein part holds in
a more general setting and will be proved separately.

\begin{theorem}
\label{thm:conformal-geometric-rigidity}
Let \((X^n,g)\), \(n\geq 3\), be the universal Riemannian cover of a closed
Riemannian manifold, and let \(\lambda_0=\lambda_0(X,g)\). Assume that
\[
        \operatorname{scal}_g\equiv -\frac{4n}{n-1}\lambda_0.
\]
Suppose that there exists a generalized ground state \(f>0\) such that the
conformal metric
\[
        \widehat g:=f^{4/(n-1)}g
\]
is Ricci-flat. Then
\[
        \Ric_g=-\frac{4\lambda_0}{n-1}g.
\]
Furthermore, if \(\lambda_0>0\), then \((X,g)\) has constant sectional curvature
\[
        \sec_g\equiv -\frac{4\lambda_0}{(n-1)^2}.
\]
\end{theorem}

We begin with the Einstein part in the general complete setting. The key point
is that the sharp scalar identity and the ground-state equation force the
conformal scale \(\rho=f^{2/(n-1)}\) to have vanishing Hessian with respect to
the Ricci-flat metric, allowing us to apply Brinkmann's classical criterion for
conformally related Einstein metrics \cite{Brinkmann1925}.

\begin{proposition}
\label{prop:conformal-einstein}
Let \((X^n,g)\), \(n\ge3\), be a connected complete Riemannian manifold, and let
\(\lambda_0=\lambda_0(X,g)\) be the bottom of the \(L^2\)-spectrum of the
nonnegative scalar Laplacian. Assume that
\[
        \scal_g\equiv -\frac{4n}{n-1}\lambda_0.
\]
Suppose that there exists a generalized ground state \(f>0\) such that
\[
        \widehat g:=f^{4/(n-1)}g
\]
is Ricci-flat. Then
\[
        \Ric_g=-\frac{4\lambda_0}{n-1}g.
\]
Moreover, if
\[
        \rho:=f^{2/(n-1)},
        \qquad
        k^2:=\frac{4\lambda_0}{(n-1)^2},
\]
then
\begin{equation}
\label{eq:rho-gradient-harmonic-hessian}
        |d\rho|_{\widehat g}^2=k^2,
        \qquad
        \Delta_{\widehat g}\rho=0,
        \qquad
        \nabla_{\widehat g}^2\rho=0.
\end{equation}
\end{proposition}

\begin{proof}
Set
\[
        \alpha:=\frac{n-1}{2},
        \qquad
        \rho:=f^{2/(n-1)}.
\]
Then
\[
        f=\rho^\alpha,
        \qquad
        \widehat g=\rho^2g,
        \qquad
        \lambda_0=\alpha^2k^2,
        \qquad
        \scal_g=-n(n-1)k^2.
\]
Since \(g=\rho^{-2}\widehat g\), the conformal scalar-curvature formula
\cite[Ch.~1, \S J, Theorem~1.159(f)]{Besse}, applied to
\(g=e^{2u}\widehat g\) with \(u=-\log\rho\), gives, using
\(\scal_{\widehat g}=0\),
\[
        \scal_g
        =
        -2(n-1)\rho\Delta_{\widehat g}\rho
        -n(n-1)|d\rho|_{\widehat g}^2.
\]
Hence
\begin{equation}
\label{eq:rho-scalar-equation}
        \rho\Delta_{\widehat g}\rho
        +\frac n2 |d\rho|_{\widehat g}^2
        =
        \frac n2 k^2.
\end{equation}

On the other hand, the ground-state equation for \(f=\rho^\alpha\) is
\[
        \Delta_g(\rho^\alpha)=\alpha^2k^2\rho^\alpha.
\]
By direct calculation,
\[
        \Delta_{\widehat g}(\rho^\alpha)
        =
        \alpha\rho^{\alpha-1}\Delta_{\widehat g}\rho
        -
        \alpha(\alpha-1)\rho^{\alpha-2}|d\rho|_{\widehat g}^2.
\]
Using the previous two identities and the conformal transformation for the nonnegative laplacian, we obtain
\[
\begin{aligned}
        \Delta_g(\rho^\alpha)
        &=
        \rho^2\left(
        \Delta_{\widehat g}(\rho^{\alpha})
        +(n-2)\rho^{-1}\langle d\rho,d(\rho^{\alpha})\rangle_{\widehat g}
        \right)\\
        &=
        \rho^2\left(
        \alpha\rho^{\alpha-1}\Delta_{\widehat g}\rho
        -
        \alpha(\alpha-1)\rho^{\alpha-2}|d\rho|_{\widehat g}^2
        +(n-2)\alpha\rho^{\alpha-2}|d\rho|_{\widehat g}^2
        \right)  \\
        &=
        \alpha\rho^\alpha
        \left(
        \rho\Delta_{\widehat g}\rho
        +(n-1-\alpha)|d\rho|_{\widehat g}^2
        \right)\\
        &=
        \alpha\rho^\alpha
        \left(
        \rho\Delta_{\widehat g}\rho
        +\alpha |d\rho|_{\widehat g}^2
        \right),   
\end{aligned}
\]
where in the last line we used \(2\alpha=n-1\).
Comparing with the ground-state equation and dividing by
\(\alpha\rho^\alpha>0\), we get
\begin{equation}
\label{eq:rho-ground-state-equation}
        \rho\Delta_{\widehat g}\rho
        +\alpha |d\rho|_{\widehat g}^2
        =
        \alpha k^2.
\end{equation}

Subtracting \eqref{eq:rho-ground-state-equation} from \eqref{eq:rho-scalar-equation}, and using again \(2\alpha=n-1\), gives \(|d\rho|_{\widehat g}^2=k^2\).
Substituting this back into \eqref{eq:rho-ground-state-equation} gives
\[
        \Delta_{\widehat g}\rho=0.
\]

Since \(\widehat g\) is Ricci-flat, and since \(|d\rho|_{\widehat g}^2=k^2\) and \(\Delta_{\widehat g}\rho=0\), the Bochner formula for \(d\rho\), with
the nonnegative Laplacian convention, gives
\[
        0=\frac12\Delta_{\widehat g}|d\rho|_{\widehat g}^2
        =
        -|\nabla_{\widehat g}^2\rho|_{\widehat g}^2
        +\langle d\rho,d\Delta_{\widehat g}\rho\rangle_{\widehat g}
        -\Ric_{\widehat g}(\nabla^{\widehat g}\rho,\nabla^{\widehat g}\rho)
        =-|\nabla_{\widehat g}^2\rho|_{\widehat g}^2.
\]
Hence
\[
        \nabla_{\widehat g}^2\rho=0.
\]
This proves \eqref{eq:rho-gradient-harmonic-hessian}.

Since \(\widehat g\) is Ricci-flat and
\(\nabla_{\widehat g}^2\rho=0\), Brinkmann's criterion, in the formulation of
\cite[Corollary~5.3]{KuhnelRademacher2008}, implies that
\(g=\rho^{-2}\widehat g\) is Einstein. Its Einstein constant is determined by
the scalar-curvature assumption:
\[
        \Ric_g
        =
        \frac{\scal_g}{n}g
        =
        -\frac{4\lambda_0}{n-1}g.\qedhere
\]
\end{proof}

\begin{proof}[Proof of \Cref{thm:conformal-geometric-rigidity}]
The Einstein conclusion follows from \Cref{prop:conformal-einstein}. It remains
to consider the case \(\lambda_0>0\).

Set
\[
        k^2:=\frac{4\lambda_0}{(n-1)^2},
        \qquad
        \rho:=f^{2/(n-1)}.
\]
Then \(\widehat g=\rho^2g\). Since \(\widehat g\) is Ricci-flat,
\eqref{eq:rho-gradient-harmonic-hessian} gives
\[
        |d\rho|_{\widehat g}^2=k^2,
        \qquad
        \Delta_{\widehat g}\rho=0,
        \qquad
        \nabla_{\widehat g}^2\rho=0.
\]

Since \(\widehat g=\rho^2g\), the first identity gives \(|d\rho|_g^2=k^2\rho^2\).
Define
\[
        b:=-\frac{1}{k}\log\rho.
\]
Then
\[
        |db|_g=1.
\]
Next we compute the Hessian of \(b\). Since
\(\nabla_{\widehat g}^2\rho=0\), and since
\(|d\rho|_g^2=k^2\rho^2\), the formula for the conformal change of the
Hessian gives
\[
        0=\nabla_{\widehat g}^2\rho
        =
        \nabla_g^2\rho
        -2\rho^{-1}d\rho\otimes d\rho
        +\rho^{-1}|d\rho|_g^2g
        =
        \nabla_g^2\rho
        -2\rho^{-1}d\rho\otimes d\rho
        +k^2\rho g.
\]
Therefore
\[
        \nabla_g^2(\log\rho)
        =
        \rho^{-1}\nabla_g^2\rho
        -\rho^{-2}d\rho\otimes d\rho
        =
        \rho^{-2}d\rho\otimes d\rho-k^2g.
\]
Using
\[
        d\log\rho=-k\,db,
        \qquad
        \rho^{-2}d\rho\otimes d\rho=k^2db\otimes db,
\]
we get
\[
        \nabla_g^2b
        =
        k(g-db\otimes db).
\]
Taking the trace and using the nonnegative Laplacian convention gives
\[
        \Delta_g b=-(n-1)k.
\]

Now define
\[
        \varphi:=e^{-(n-1)kb}=\rho^{n-1}=f^2.
\]
Then \(\varphi>0\), and
\[
        |\nabla\log\varphi|_g=(n-1)k.
\]
Moreover,
\[
        \Delta_g\varphi
        =
        \varphi\left(\Delta_g\log\varphi-|d\log\varphi|_g^2\right).
\]
Since \(\log\varphi=-(n-1)kb\), we have
\[
        \Delta_g\log\varphi
        =
        -(n-1)k\Delta_gb
        =
        (n-1)^2k^2,
\]
and
\[
        |d\log\varphi|_g^2=(n-1)^2k^2.
\]
Hence
\[
        \Delta_g\varphi=0.
\]
Thus \(\varphi\) is a positive harmonic function satisfying
\[
        |\nabla\log\varphi|_g=(n-1)k.
\]

Therefore \cite[Theorem~6]{LedrappierWang} applies, after possibly constantly scaling \(g\). 
It follows that \((X,g)\) has constant sectional curvature
\[
        -k^2=-\frac{4\lambda_0}{(n-1)^2}.
\]
This proves the theorem.
\end{proof}

\subsection{Equality defects}
\label{subsec:equality-defects}

From now on we assume the hypotheses of \Cref{thm:einstein-rigidity}.
Replacing each \(\psi_j\) by \(\psi_j/\|\psi_j\|_{L^2}\), we assume without
loss of generality that
\begin{equation}
\label{eq:section5-normalization}
        \|\psi_j\|_{L^2}=1.
\end{equation}
Set
\[
        r_j:=|\psi_j|,
        \qquad
        \alpha_n:=\frac n{n-1},
        \qquad
        \tau_j:=\|D_{s_j}\psi_j\|_{L^2},
        \qquad
        \delta_j:=s_j+\tau_j.
\]
After discarding finitely many terms, we may assume that \(\tau_j\leq1\)
for every \(j\).

\begin{proposition}
\label{prop:equality-defects}
Let \(f>0\) be a generalized ground state.
Then there exists a constant \(C>0\), independent of \(j\), such that
\begin{equation}
\label{eq:Pf-sharp-defect}
        \int_X |P^{s_j}_f(\psi_j)|^2\dmu
        \le C\delta_j .
\end{equation}
\end{proposition}

\begin{proof}
Recall that
\[
        \scal_g\equiv -4\alpha_n\lambda_0,
        \qquad \alpha_n:=\frac{n}{n-1}.
\]
The integrated Lichnerowicz formula and \eqref{eq:Romega-bound} give a
uniform bound for \(\|\nabla^{s_j}\psi_j\|_{L^2}\), hence also for
\(\|dr_j\|_{L^2}\). More precisely,
\[
        \left|
        \|\nabla^{s_j}\psi_j\|_{L^2}^2-\alpha_n\lambda_0
        \right|
        \leq \tau_j^2+C_\omega s_j.
\]
The Dirac term in \eqref{eq:refined-kato-identity} satisfies
\[
        \left|
        \frac2{n-1}\operatorname{Re}
        \int_X
        \langle D_{s_j}\psi_j,c(\nabla^g\log r_j)\psi_j\rangle\dmu
        \right|
        \leq C\tau_j\|dr_j\|_{L^2}
        \leq C\tau_j.
\]
Therefore, integrating \eqref{eq:refined-kato-identity}, we obtain
\begin{equation}
\label{eq:two-defects-identity}
        \int_X |P^{s_j}_{r_j}(\psi_j)|^2\dmu
        +
        \alpha_n\left(
        \int_X |dr_j|^2\dmu
        -\lambda_0\int_X r_j^2\dmu
        \right) 
        \leq C\bigl(s_j+\tau_j+\tau_j^2\bigr).
\end{equation}
The two terms on the left-hand side are nonnegative: the first by definition, and the second by Rayleigh's inequality \eqref{eq:rayleigh-inequality}. 
Since \(\tau_j\leq1\), after enlarging \(C\) if necessary, we obtain
\begin{equation}
\label{eq:Prj-sharp-defect-aux}
        \int_X |P^{s_j}_{r_j}(\psi_j)|^2\dmu
        \le C\delta_j
\end{equation}
and
\begin{equation}
\label{eq:rayleigh-sharp-defect-aux}
        0\le
        \int_X |dr_j|^2\dmu
        -\lambda_0\int_X r_j^2\dmu
        \le C\delta_j.
\end{equation}

Next we compare \(r_j\) with the fixed generalized ground state \(f\). By the
ground-state transform applied to \(v=r_j\),
\[
        \int_X f^2\left|d\left(\frac{r_j}{f}\right)\right|^2\dmu
        =
        \int_X |dr_j|^2\dmu
        -\lambda_0\int_X r_j^2\dmu .
\]
Together with \eqref{eq:rayleigh-sharp-defect-aux}, this gives
\begin{equation}
\label{eq:ground-state-ratio-sharp-defect-aux}
        \int_X f^2\left|d\left(\frac{r_j}{f}\right)\right|^2\dmu
        \le C\delta_j .
\end{equation}

It remains to pass from the \(r_j\)-Kato defect to the \(f\)-Kato defect. Put
\[
        h_j:=\frac{r_j}{f}.
\]
On the set \(\{r_j>0\}\), we have
\[
        P^{s_j}_f(\psi_j)-P^{s_j}_{r_j}(\psi_j)
        =
        \frac n{n-1}d\log h_j\otimes\psi_j
        +
        \frac1{n-1}c(\,\cdot\,)c(\nabla\log h_j)\psi_j .
\]
Hence, for a dimensional constant \(C_n>0\),
\[
        |P^{s_j}_f(\psi_j)-P^{s_j}_{r_j}(\psi_j)|^2
        \le
        C_n r_j^2|d\log h_j|^2
        =
        C_nf^2|dh_j|^2 .
\]
By the convention set for \(P^{s_j}_{r_j}\) in \Cref{subsec:refined-kato-defect}, and since \(\nabla^{s_j}\psi_j=0\) almost everywhere on the zero set of \(r_j\), this inequality holds almost everywhere on \(X\). 
Therefore, using \eqref{eq:Prj-sharp-defect-aux} and \eqref{eq:ground-state-ratio-sharp-defect-aux}, we find \(C>0\) such that
\[
        \int_X |P^{s_j}_f(\psi_j)|^2\dmu
        \le
        2\int_X |P^{s_j}_{r_j}(\psi_j)|^2\dmu
        +
        2C_n\int_X f^2|dh_j|^2\dmu
        \le C\delta_j,
\]
This proves \eqref{eq:Pf-sharp-defect}.
\end{proof}

\subsection{Conformal spinors}
\label{subsec:conformal-spinors}

We record the conformal spinor identity in the form used below.

\begin{lemma}
\label{lem:conformal-spinor-identity}
Let \(q>0\) be smooth and set
\[
        \widehat g=q^{4/(n-1)}g.
\]
Let
\[
        \beta_q:\Sigma_gX\longrightarrow\Sigma_{\widehat g}X
\]
be the standard conformal spinor identification. Let \(E\to X\) be a
Hermitian vector bundle with unitary connection. For
\(\psi\in C^\infty(X,\Sigma_gX\otimes E)\), set
\[
        \widehat\psi:=q^{-1}(\beta_q\otimes\id_E)\psi.
\]
Then
\begin{equation}
\label{eq:conformal-spinor-identity}
        \nabla^{\Sigma_{\widehat g}X\otimes E}\widehat\psi
        =
        q^{-1}(\beta_q\otimes\id_E)(P_q^E\psi).
\end{equation}
Moreover,
\begin{equation}
\label{eq:conformal-mass-density}
        |\widehat\psi|_{\widehat g}^2\,d\mu_{\widehat g}
        =
        q^{2/(n-1)}|\psi|_g^2\,d\mu_g,
\end{equation}
and
\begin{equation}
\label{eq:conformal-derivative-density}
        |\nabla^{\Sigma_{\widehat g}X\otimes E}\widehat\psi|_{\widehat g}^2\,d\mu_{\widehat g}
        =
        q^{-2/(n-1)}|P_q^E\psi|_g^2\,d\mu_g.
\end{equation}
\end{lemma}

\begin{proof}
The conformal spinorial connection formula for \(\widehat g=q^{4/(n-1)}g\) gives
\begin{align*}
        \nabla^{\Sigma_{\widehat g}X\otimes E}_Y
        \bigl((\beta_q\otimes\id_E)\psi\bigr)
        &=(\beta_q\otimes\id_E)
        \Bigl(
        \nabla^{\Sigma_gX\otimes E}_Y\psi
        -\frac1{n-1}Y(\log q)\psi \\
        &\hspace{3.7cm}
        -\frac1{n-1}c_g(Y)c_g(\nabla^g\log q)\psi
        \Bigr),
\end{align*}
see \cite[Proposition~2.33]{BourguignonHijaziMilhoratMoroianuMoroianu}.
We now differentiate the scalar factor \(q^{-1}\). 
For every vector field \(Y\),
\[
\begin{aligned}
\nabla_Y^{\Sigma_{\widehat g}X\otimes E}\widehat\psi
&=
-Y(\log q)\,q^{-1}(\beta_q\otimes\mathrm{id}_E)\psi
+q^{-1}
\nabla_Y^{\Sigma_{\widehat g}X\otimes E}
\bigl((\beta_q\otimes \mathrm{id}_E)\psi\bigr) \\
&=
q^{-1}(\beta_q\otimes\mathrm{id}_E)
\left(
\nabla_Y^{\Sigma_gX\otimes E}\psi
-\frac{n}{n-1}Y(\log q)\psi
-\frac{1}{n-1}c_g(Y)c_g(\nabla^g\log q)\psi
\right) \\
&=
q^{-1}(\beta_q\otimes\mathrm{id}_E)\bigl(P_q^E(\psi)(Y)\bigr).
\end{aligned}
\]
This proves \eqref{eq:conformal-spinor-identity}. 
The density identities \eqref{eq:conformal-mass-density} and \eqref{eq:conformal-derivative-density} follow from the fiberwise unitarity of \(\beta_q\), the relation
\[
        d\mu_{\widehat g}=q^{2n/(n-1)}d\mu_g,
\]
and the fact that one-form norms scale by \(q^{-2/(n-1)}\).
\end{proof}

\subsection{Recentering}
\label{subsec:equivariant-recentering}

Fix a generalized ground state \(f>0\), whose existence is guaranteed by \Cref{prop:ground-state-existence}. 
Let \(\Gamma\) be the deck group of \(X\to M\). 
Fix a finite symmetric generating set of \(\Gamma\) and write \(|\gamma|\) for the associated word length.
For \(N\in\mathbb N_0\), and with \(e\in\Gamma\) denoting the identity element, set
\[
        B_N(e):=\{\gamma\in\Gamma:|\gamma|\le N\},
        \qquad
        Q_N:=|B_N(e)|,
        \qquad
        B_N(\gamma):=\gamma B_N(e).
\]
Choose a measurable fundamental domain
\[
        \mathcal F\subset X
\]
with compact closure and boundary of measure zero, and fix \(x_0\in\mathcal F\).

For each \(j\), define the cell mass and the \(f\)-conformal Kato cell defect by
\begin{equation}
\label{eq:cell-mass-defect}
        m_j(\gamma):=\int_{\gamma\mathcal F}r_j^2\dmu,
        \qquad
        e_j(\gamma):=\int_{\gamma\mathcal F}|P_f^{s_j}(\psi_j)|^2\dmu.
\end{equation}
Then
\begin{equation}
\label{eq:cell-mass-defect-sums}
        \sum_{\gamma\in\Gamma}m_j(\gamma)=1,
        \qquad
        \sum_{\gamma\in\Gamma}e_j(\gamma)\le C\delta_j
\end{equation}
by \eqref{eq:section5-normalization} and \eqref{eq:Pf-sharp-defect}.

For \(N\in\mathbb N_0\), define the \(N\)-local cell mass and the
\(N\)-local \(f\)-conformal Kato defect energy centered at \(\gamma\) by
\begin{equation}
\label{eq:local-mass-defect-definitions}
        M_{j,N}(\gamma):=
        \sum_{\zeta\in B_N(\gamma)}m_j(\zeta),
        \qquad
        E_{j,N}(\gamma):=
        \sum_{\zeta\in B_N(\gamma)}e_j(\zeta).
\end{equation}

The purpose of the recentering argument is to prevent the \(L^2\)-mass of the spinors from escaping to infinity on the universal cover. 
The quantity \(m_j(\gamma)\) measures the mass of \(\psi_j\) on the cell \(\gamma\mathcal F\), while \(e_j(\gamma)\) measures the \(f\)-conformal Kato defect on the same cell. 
We choose a cell \(\gamma_j\mathcal F\) such that, for each fixed word-metric radius $N$, the mass and defect energy in the union of cells within word distance $N$ of \(\gamma_j\mathcal F\) are controlled by the mass of the chosen cell itself.
After pulling this cell back to the fixed fundamental domain and normalizing, the global estimates may no longer hold, but the local estimates needed for compactness survive.

\begin{lemma}
\label{lem:discrete-recentering}
There exists a sequence \((\gamma_j)_{j\ge1}\subset\Gamma\) such that \(m_j(\gamma_j)>0\), and, for every \(0\le N\le j\),
\begin{equation}
\label{eq:discrete-recentered-mass}
        M_{j,N}(\gamma_j)
        \le
        A_Nm_j(\gamma_j),
\end{equation}
and
\begin{equation}
\label{eq:discrete-recentered-defect}
        E_{j,N}(\gamma_j)
        \le
        A_N\delta_jm_j(\gamma_j),
\end{equation}
where
\[
        A_N:=2^{N+4}(1+C)Q_N.
\]
In particular, for every fixed \(N\in\mathbb N_0\), the estimates \eqref{eq:discrete-recentered-mass} and \eqref{eq:discrete-recentered-defect} hold for all sufficiently large \(j\).
\end{lemma}

\begin{proof}
Fix \(j\ge1\). Throughout the proof, the quantities
\(m_j,e_j,M_{j,N},E_{j,N}\) are those associated with this fixed \(j\).

For \(0\le N\le j\), define the bad mass set
\[
        \mathcal B^M_{j,N}
        :=
        \left\{
        \gamma\in\Gamma:
        m_j(\gamma)>0,\ 
        M_{j,N}(\gamma)>A_Nm_j(\gamma)
        \right\}.
\]
We have
\begin{align*}
        \sum_{\gamma\in\mathcal B^M_{j,N}}m_j(\gamma)
        &\le
        \frac1{A_N}
        \sum_{\gamma\in\mathcal B^M_{j,N}}M_{j,N}(\gamma) \\
        &\le
        \frac1{A_N}
        \sum_{\gamma:m_j(\gamma)>0}
        \sum_{\zeta\in B_N(\gamma)}m_j(\zeta) \\
        &=
        \frac1{A_N}
        \sum_{\zeta\in\Gamma}m_j(\zeta)
        \#\{\gamma:m_j(\gamma)>0,\ \zeta\in B_N(\gamma)\}.
\end{align*}
where in the last identity we interchanged the order of summation as all summands are nonnegative.

For fixed \(\zeta\), the condition
\[
        \zeta\in B_N(\gamma)=\gamma B_N(e)
\]
is equivalent to the existence of \(\beta\in B_N(e)\) such that
\(\zeta=\gamma\beta\), or equivalently \(\gamma=\zeta\beta^{-1}\). Since the
generating set is symmetric, \(B_N(e)^{-1}=B_N(e)\). 
Thus, for fixed \(\zeta\), there are exactly \(Q_N=|B_N(e)|\) elements \(\gamma\in\Gamma\) such that \(\zeta\in B_N(\gamma)\), and at most \(Q_N\) of them can also satisfy \(m_j(\gamma)>0\).
Using \eqref{eq:cell-mass-defect-sums}, we get
\[
        \sum_{\gamma\in\mathcal B^M_{j,N}}m_j(\gamma)
        \le
        \frac{Q_N}{A_N}
        \sum_{\zeta\in\Gamma}m_j(\zeta)
        =
        \frac{Q_N}{A_N}
        \le
        2^{-N-4}.
\]

Similarly, define the bad defect set
\[
        \mathcal B^E_{j,N}
        :=
        \left\{
        \gamma\in\Gamma:
        m_j(\gamma)>0,\ 
        E_{j,N}(\gamma)>A_N\delta_jm_j(\gamma)
        \right\}.
\]
If \(\delta_j=0\), then \(e_j\equiv0\) and the bad defect set is empty.
If \(\delta_j>0\), the same counting argument and
\eqref{eq:cell-mass-defect-sums} give
\begin{align*}
        \sum_{\gamma\in\mathcal B^E_{j,N}}m_j(\gamma)
        \le
        \frac1{A_N\delta_j}
        \sum_{\gamma:m_j(\gamma)>0}E_{j,N}(\gamma)
        \le
        \frac{CQ_N}{A_N}
        \le
        2^{-N-4}.
\end{align*}

Define the total bad set for the fixed index \(j\) by
\[
        \mathcal B_j
        :=
        \bigcup_{N=0}^j
        \left(\mathcal B^M_{j,N}\cup\mathcal B^E_{j,N}\right).
\]
Then
\[
        \sum_{\gamma\in\mathcal B_j}m_j(\gamma)
        \le
        \sum_{N=0}^j
        \left(2^{-N-4}+2^{-N-4}\right)
        <1.
\]
Since \(\sum_{\gamma\in\Gamma}m_j(\gamma)=1\), there exists \(\gamma_j\notin\mathcal B_j\) with \(m_j(\gamma_j)>0\). 
For this choice, both \eqref{eq:discrete-recentered-mass} and \eqref{eq:discrete-recentered-defect} hold for every \(0\le N\le j\).
Choosing such a \(\gamma_j\) for every \(j\) gives the desired sequence.
\end{proof}

\begin{lemma}
\label{lem:recentered-ground-states}
Let \((\gamma_j)_{j\geq1}\subset\Gamma\) be any sequence, and define
\[
        f_j(x):=\frac{f(\gamma_jx)}{f(\gamma_jx_0)}.
\]
Then each \(f_j\) is a generalized ground state satisfying \(f_j(x_0)=1\).
Moreover, after passing to a subsequence, there exists a generalized ground
state \(f_\infty\in C^\infty(X)\) such that
\[
        f_j\to f_\infty
        \quad\text{in }C^\infty_\loc(X),
        \qquad
        f_\infty(x_0)=1.
\]
\end{lemma}

\begin{proof}
First we show that the functions \(f_j\) are again normalized generalized
ground states. Since each deck transformation \(\gamma\in\Gamma\) is an
isometry of \((X,g)\), the Laplace--Beltrami operator commutes with pullback by
\(\gamma\); see Chavel \cite[Chapter~II, \S1, Eq.~(2), p.~27]{ChavelEigenvalues}.
Thus
\[
        \Delta_g(f\circ\gamma_j)
        =
        (\Delta_g f)\circ\gamma_j
        =
        \lambda_0(f\circ\gamma_j).
\]
Since \(f(\gamma_jx_0)>0\) is a constant, it follows that
\[
        \Delta_g f_j=\lambda_0 f_j,
        \qquad
        f_j>0,
        \qquad
        f_j(x_0)=1.
\]

We next prove uniform local \(C^0\)-bounds. Let \(K\Subset X\). Since \(X\) is
connected, one can join \(x_0\) to \(K\) by finitely many relatively compact
coordinate balls. Cover \(K\) by finitely many coordinate
balls \(B_1,\ldots,B_N\). 
Choose a finite Harnack chain from \(x_0\) to \(B_i\), by which we mean a
finite sequence of relatively compact coordinate balls \(B_{i,1},\ldots,B_{i,p_i}\) such that \(x_0\in B_{i,1}\), \(B_{i,p_i}=B_i\), and
\(B_{i,r}\cap B_{i,r+1}\neq\emptyset\) for \(r=1,\ldots,p_i-1\).
Set \(\Omega_i:=\bigcup_{r=1}^{p_i}B_{i,r}\).
On each ball \(B_{i,r}\), the local Harnack inequality
\cite[Theorem~8.20]{GilbargTrudinger}, applied in local coordinates to the
positive solutions of
\[
        (\Delta_g-\lambda_0)f_j=0,
\]
gives a constant \(H_{i,r}\geq1\), independent of \(j\), such that
\[
        \sup_{B_{i,r}} f_j\leq H_{i,r}\inf_{B_{i,r}} f_j .
\]
If \(y\in B_{i,a}\) and \(z\in B_{i,b}\), then, choosing points in the nonempty overlaps
\(B_{i,r}\cap B_{i,r+1}\) and applying the Harnack inequalities successively along the subchain
joining \(B_{i,a}\) to \(B_{i,b}\), one obtains
\[
        f_j(y)\le \left(\prod_{r=1}^{p_i} H_{i,r}\right) f_j(z).
\]
Taking the supremum over \(y\in \Omega_i\) and the infimum over \(z\in \Omega_i\) gives
\[
        \sup_{\Omega_i} f_j
        \le
        \left(\prod_{r=1}^{p_i} H_{i,r}\right)
        \inf_{\Omega_i} f_j .
\]
Since \(f_j(x_0)=1\), we obtain 
\[
        \left(\prod_{r=1}^{p_i}H_{i,r}\right)^{-1}
        \leq
        f_j
        \leq
        \prod_{r=1}^{p_i}H_{i,r}
        \qquad\text{on }\Omega_i .
\]
Taking the maximum over \(i=1,\ldots,N\), we find a constant \(C_K\geq1\),
independent of \(j\), such that
\[
        C_K^{-1}\leq f_j\leq C_K
        \qquad\text{on }K.
\]

We now obtain higher-order estimates. 
Let \(K\Subset X\), and choose an open set \(\Omega\Subset X\) with \(K\subset\Omega\). 
Applying the previous Harnack-chain bound to \(\overline{\Omega}\), we obtain a uniform \(C^0\)-bound for \(f_j\) on \(\Omega\). 
In local coordinates, the equation
\[
        (\Delta_g-\lambda_0)f_j=0
\]
is a homogeneous linear uniformly elliptic equation with smooth coefficients.
The interior Schauder estimates \cite[Theorem~6.2]{GilbargTrudinger} imply that, for every \(\alpha\in(0,1)\),
\[
        \|f_j\|_{C^{2,\alpha}(K)}
        \leq
        C_{K,\Omega,\alpha}\|f_j\|_{C^0(\Omega)}
        \leq
        C'_{K,\Omega,\alpha},
\]
with constants independent of \(j\). 
Bootstrapping on nested relatively compact open sets gives, for every \(k\geq0\), a constant \(C_{K,k}\), independent of \(j\), such that
\[
        \|f_j\|_{C^k(K)}\leq C_{K,k}.
\]

Choose a compact exhaustion \(K_1\subset K_2\subset \cdots \subset X\), with \(K_m\subset \operatorname{int}(K_{m+1})\), \(\bigcup_m K_m=X\).
The preceding estimates give uniform \(C^k(K_m)\)-bounds for every \(m\) and every \(k\). 
By Arzelà--Ascoli and a diagonal argument, after passing to a subsequence, the functions \(f_j\) converge in \(C^\infty_\loc(X)\)
to a smooth function \(f_\infty\). 
Passing to the limit in
\[
        \Delta_g f_j=\lambda_0 f_j,
        \qquad
        f_j(x_0)=1,
\]
we obtain
\[
        \Delta_g f_\infty=\lambda_0 f_\infty,
        \qquad
        f_\infty(x_0)=1.
\]
Finally, the local lower bounds pass to the limit: for every compact
\(K\Subset X\), there is \(C_K\geq1\) such that \(f_j\geq C_K^{-1}\) on \(K\).
Hence \(f_\infty>0\). Therefore \(f_\infty\) is a generalized ground state.
\end{proof}

\begin{lemma}
\label{lem:kato-defect-equivariance}
Let \(q>0\) be smooth, let \(s\in\R\), and let \(\gamma\in\Gamma\). Choose a
lift \(\Phi_{\gamma,s}\in\Gamma_s\), \(\sigma_s(\Phi_{\gamma,s})=\gamma\),
and write \(U_{\gamma,s}^{(k)}\) for the induced action on
\(\Lambda^kT^*X\otimes\Sigma_gX\otimes F\); see
\Cref{subsec:Gromov-central-index}. Then, for every smooth twisted spinor
\(\varphi\),
\begin{equation}
\label{eq:kato-defect-equivariance}
        P^s_{q\circ\gamma^{-1}}(U_{\gamma,s}^{(0)}\varphi)
        =
        U_{\gamma,s}^{(1)}(P_q^s\varphi).
\end{equation}
Equivalently,
\begin{equation}
\label{eq:kato-defect-equivariance-inverse}
        P^s_{q\circ\gamma}(U_{\gamma^{-1},s}^{(0)}\varphi)
        =
        U_{\gamma^{-1},s}^{(1)}(P_q^s\varphi).
\end{equation}
\end{lemma}

\begin{proof}
The connection term follows from \Cref{prop:gromov-equivariance}:
\[
        \nabla^s(U_{\gamma,s}^{(0)}\varphi)
        =
        U_{\gamma,s}^{(1)}(\nabla^s\varphi).
\]
The scalar terms follow from the chain rule
\[
        d\log(q\circ\gamma^{-1})=(\gamma^{-1})^*d\log q
\]
and from the fact that the deck action is isometric and preserves Clifford
multiplication. Substituting these identities into the definition of \(P_q^s\)
gives \eqref{eq:kato-defect-equivariance}. Replacing \(\gamma\) by
\(\gamma^{-1}\) gives \eqref{eq:kato-defect-equivariance-inverse}.
\end{proof}

We now apply \Cref{lem:discrete-recentering} to the functions \(m_j\) and
\(e_j\) defined in \eqref{eq:cell-mass-defect}, and fix a corresponding
sequence \((\gamma_j)\subset\Gamma\). By \Cref{lem:recentered-ground-states},
after passing to a subsequence we may assume
\[
        f_j\to f_\infty
        \qquad\text{in }C^\infty_{\loc}(X),
\]
where \(f_\infty>0\) is a generalized ground state.
Choose a lift
\[
        \Phi_{\gamma_j^{-1},s_j}\in\Gamma_{s_j},
        \qquad
        \sigma_{s_j}(\Phi_{\gamma_j^{-1},s_j})=\gamma_j^{-1},
\]
and write
\[
        U_{\gamma_j^{-1},s_j}^{(k)}
        :=
        U_{\Phi_{\gamma_j^{-1},s_j}}^{(k)},
        \qquad
        k=0,1.
\]
Set
\begin{equation}
\label{eq:recentered-spinors}
        \widetilde\psi_j
        :=
        U_{\gamma_j^{-1},s_j}^{(0)}\psi_j,
        \qquad
        a_j:=m_j(\gamma_j)^{-1/2},
        \qquad
        \Phi_j:=a_j\widetilde\psi_j.
\end{equation}

\begin{lemma}
\label{lem:recentered-local-estimates}
For every compact set \(K\Subset X\), there exists \(C_K>0\) such that, for
all sufficiently large \(j\),
\begin{equation}
\label{eq:recentered-mass-normalization}
        \int_{\mathcal F}|\Phi_j|^2\dmu=1,
\end{equation}
\begin{equation}
\label{eq:recentered-local-mass-bound}
        \int_K |\Phi_j|^2\dmu\le C_K,
\end{equation}
and
\begin{equation}
\label{eq:recentered-local-Pfj-bound}
        \int_K |P^{s_j}_{f_j}(\Phi_j)|^2\dmu
        \le
        C_K\delta_j.
\end{equation}
\end{lemma}

\begin{proof}
By the norm identity from \Cref{prop:gromov-equivariance},
\[
        |\widetilde\psi_j|(x)
        =
        |\psi_j|(\gamma_jx)
        =
        r_j(\gamma_jx).
\]
Therefore
\[
        \int_{\mathcal F}|\widetilde\psi_j|^2\dmu
        =
        \int_{\gamma_j\mathcal F}r_j^2\dmu
        =
        m_j(\gamma_j).
\]
Multiplying by \(a_j^2=m_j(\gamma_j)^{-1}\) gives
\eqref{eq:recentered-mass-normalization}.

Let \(K\Subset X\). Since \(\overline{\mathcal F}\) is compact and the deck
action is properly discontinuous, the set
\[
        S_K:=\{\delta\in\Gamma:K\cap\delta\overline{\mathcal F}\ne\emptyset\}
\]
is finite. Choose \(N=N(K)\) such that
\[
        S_K\subset B_N(e).
\]
The \(\Gamma\)-translates of \(\mathcal F\) cover \(X\) up to a null set, and \(\partial\mathcal F\) has measure zero. 
Hence, using the norm identity again,
\begin{align*}
        \int_K|\Phi_j|^2\dmu
        &\le
        a_j^2
        \sum_{\delta\in S_K}
        \int_{\delta\mathcal F}|\widetilde\psi_j|^2\dmu \\
        &\le
        a_j^2
        \sum_{\delta\in B_N(e)}
        \int_{\delta\mathcal F}r_j^2(\gamma_jx)\dmu(x) \\
        &=
        a_j^2
        \sum_{\delta\in B_N(e)}
        \int_{\gamma_j\delta\mathcal F}r_j^2\dmu \\
        &=
        a_j^2 M_{j,N}(\gamma_j).
\end{align*}
For this fixed \(N=N(K)\), \eqref{eq:discrete-recentered-mass} holds for all
sufficiently large \(j\). Thus
\[
        \int_K|\Phi_j|^2\dmu
        \le
        a_j^2A_Nm_j(\gamma_j)
        =
        A_N.
\]
This proves \eqref{eq:recentered-local-mass-bound}.

It remains to prove the defect estimate. Since
\[
        f_j=\frac{f\circ\gamma_j}{f(\gamma_jx_0)},
\]
the functions \(f_j\) and \(f\circ\gamma_j\) have the same logarithmic
differential. Therefore, by \Cref{lem:kato-defect-equivariance}, applied with
\(q=f\) and \(\gamma=\gamma_j\),
\[
        P^{s_j}_{f_j}(\widetilde\psi_j)
        =
        U_{\gamma_j^{-1},s_j}^{(1)}(P_f^{s_j}\psi_j).
\]
Since \(P^{s_j}_{f_j}\) is linear in the spinor variable,
\[
        P^{s_j}_{f_j}(\Phi_j)
        =
        a_jU_{\gamma_j^{-1},s_j}^{(1)}(P_f^{s_j}\psi_j).
\]
Using the norm identity for \(k=1\), we obtain
\begin{align*}
        \int_K |P^{s_j}_{f_j}(\Phi_j)|^2\dmu
        &\le
        a_j^2
        \sum_{\delta\in S_K}
        \int_{\delta\mathcal F}
        \left|
        U_{\gamma_j^{-1},s_j}^{(1)}(P_f^{s_j}\psi_j)
        \right|^2
        \dmu \\
        &\le
        a_j^2
        \sum_{\delta\in B_N(e)}
        \int_{\gamma_j\delta\mathcal F}
        |P_f^{s_j}(\psi_j)|^2\dmu \\
        &=
        a_j^2E_{j,N}(\gamma_j).
\end{align*}
For this fixed \(N=N(K)\), \eqref{eq:discrete-recentered-defect} holds for
all sufficiently large \(j\). Therefore
\[
        \int_K |P^{s_j}_{f_j}(\Phi_j)|^2\dmu
        \le
        a_j^2A_N\delta_jm_j(\gamma_j)
        =
        A_N\delta_j.
\]
This proves \eqref{eq:recentered-local-Pfj-bound}.
\end{proof}

\begin{remark}
\label{rem:recentered-normalization-local}
The twisted spinors \(\widetilde\psi_j\) are obtained from \(\psi_j\) by
equivariant deck translation. By \Cref{prop:gromov-equivariance},
\[
        D_{s_j}\widetilde\psi_j
        =
        U_{\gamma_j^{-1},s_j}^{(0)}(D_{s_j}\psi_j).
\]
Since the translation operator is unitary, it follows that
\[
        \|\widetilde\psi_j\|_{L^2}=1,
        \qquad
        \|D_{s_j}\widetilde\psi_j\|_{L^2}=\tau_j\longrightarrow0.
\]
The additional normalization
\[
        \Phi_j=m_j(\gamma_j)^{-1/2}\widetilde\psi_j
\]
forces unit mass on the fixed fundamental domain \(\mathcal F\), but the factor
\(a_j=m_j(\gamma_j)^{-1/2}\) may be large, and
\(\|D_{s_j}\Phi_j\|_{L^2}=a_j\tau_j\) need not tend to zero.
Thus \((\Phi_j)\) should not be viewed as a globally almost harmonic sequence.
What survives, and what is needed below, is precisely the local information:
\[
        \int_K|\Phi_j|^2\dmu\le C_K,
        \qquad
        \int_K|P^{s_j}_{f_j}(\Phi_j)|^2\dmu\le C_K\delta_j.
\]
\end{remark}

\subsection{The recentered conformal limit}
\label{subsec:recentered-conformal-limit}

We now pass to the limiting conformal metric
\[
        \widehat g_\infty:=f_\infty^{4/(n-1)}g.
\]
Let
\[
        \beta_\infty:=\beta_{f_\infty}:
        \Sigma_gX\longrightarrow\Sigma_{\widehat g_\infty}X
\]
be the conformal spinor identification, and set
\begin{equation}
\label{eq:conformal-recentered-spinors}
        \widehat\Phi_j
        :=
        f_\infty^{-1}(\beta_\infty\otimes\id_F)\Phi_j
        \in
        C^\infty(X,\Sigma_{\widehat g_\infty}X\otimes F).
\end{equation}
The connection on \(\Sigma_{\widehat g_\infty}X\otimes F\) induced by
\(d+is_j\eta\) on \(F\) is denoted by \(\nabla^{\widehat g_\infty,s_j}\).

\begin{lemma}
\label{lem:fixed-limit-defect}
For every compact set \(K\Subset X\),
\begin{equation}
\label{eq:fixed-limit-Pf-defect}
        \int_K |P^{s_j}_{f_\infty}(\Phi_j)|^2\dmu
        \longrightarrow0.
\end{equation}
\end{lemma}

\begin{proof}
Let \(K\Subset X\). 
By \eqref{eq:recentered-local-Pfj-bound},
\[
        \int_K |P^{s_j}_{f_j}(\Phi_j)|^2\dmu
        \le C_K\delta_j
        \longrightarrow0.
\]
We compare \(P^{s_j}_{f_\infty}(\Phi_j)\) with \(P^{s_j}_{f_j}(\Phi_j)\). 
From the definition of \(P_q^s\), for every vector field \(Y\),
\[
        \bigl(P^{s_j}_{f_\infty}(\Phi_j)-P^{s_j}_{f_j}(\Phi_j)\bigr)(Y)
        =
        -\frac n{n-1}Y\!\left(\log\frac{f_\infty}{f_j}\right)\Phi_j
        -\frac1{n-1}c(Y)c\!\left(\nabla\log\frac{f_\infty}{f_j}\right)\Phi_j .
\]
Hence, for a dimensional constant \(C_n>0\),
\[
        |P^{s_j}_{f_\infty}(\Phi_j)-P^{s_j}_{f_j}(\Phi_j)|^2
        \le
        C_n\left|d\log\frac{f_\infty}{f_j}\right|^2|\Phi_j|^2.
\]
Since \(f_j\to f_\infty\) in \(C^\infty_{\loc}\) and \(f_\infty>0\),
\[
        d\log\frac{f_\infty}{f_j}\longrightarrow0
        \qquad\text{in }L^\infty(K).
\]
Using the local mass bound \eqref{eq:recentered-local-mass-bound}, we get
\[
        \int_K|P^{s_j}_{f_\infty}(\Phi_j)-P^{s_j}_{f_j}(\Phi_j)|^2\dmu
        \longrightarrow0.
\]
Together with the estimate \eqref{eq:recentered-local-Pfj-bound}, this proves \eqref{eq:fixed-limit-Pf-defect}.
\end{proof}

\begin{proposition}
\label{prop:recentered-conformal-limit}
After passing to a further subsequence, there exists a nonzero smooth spinor
\[
        0\ne\widehat\Phi_\infty
        \in
        C^\infty(X,\Sigma_{\widehat g_\infty}X\otimes F)
\]
such that
\[
        \nabla^{\widehat g_\infty,0}\widehat\Phi_\infty=0.
\]
\end{proposition}

\begin{proof}
Let \(K\Subset X\). By the identity \eqref{eq:conformal-mass-density}, applied with \(q=f_\infty\),
\[
        \int_K|\widehat\Phi_j|_{\widehat g_\infty}^2\,d\mu_{\widehat g_\infty}
        =
        \int_K f_\infty^{2/(n-1)}|\Phi_j|^2\dmu.
\]
Since \(f_\infty\) is bounded above on \(K\), the local mass bound \eqref{eq:recentered-local-mass-bound} gives a uniform local \(L^2\)-bound
for \(\widehat\Phi_j\).

Similarly, by the derivative-density identity
\eqref{eq:conformal-derivative-density},
\[
        \int_K|\nabla^{\widehat g_\infty,s_j}\widehat\Phi_j|_{\widehat g_\infty}^2\,d\mu_{\widehat g_\infty}
        =
        \int_K f_\infty^{-2/(n-1)}|P^{s_j}_{f_\infty}(\Phi_j)|^2\dmu.
\]
By \eqref{eq:fixed-limit-Pf-defect},
\[
        \nabla^{\widehat g_\infty,s_j}\widehat\Phi_j
        \longrightarrow0
        \qquad\text{in }L^2_{\loc}.
\]
Since
\[
        \nabla^{\widehat g_\infty,s_j}
        =
        \nabla^{\widehat g_\infty,0}+is_j\eta
\]
and \(\eta\) is smooth, the local \(L^2\)-bound for \(\widehat\Phi_j\) and
\(s_j\to0\) imply
\[
        \nabla^{\widehat g_\infty,0}\widehat\Phi_j
        \longrightarrow0
        \qquad\text{in }L^2_{\loc}.
\]
Thus \((\widehat\Phi_j)\) is bounded in \(W^{1,2}_{\loc}\). By Rellich
compactness and a diagonal argument, after passing to a subsequence,
\[
        \widehat\Phi_j\to\widehat\Phi_\infty
        \quad\text{strongly in }L^2_{\loc},
        \qquad
        \widehat\Phi_j\rightharpoonup\widehat\Phi_\infty
        \quad\text{weakly in }W^{1,2}_{\loc}.
\]
Passing to the weak limit in
\[
        \nabla^{\widehat g_\infty,0}\widehat\Phi_j\to0
        \quad\text{in }L^2_{\loc}
\]
gives
\[
        \nabla^{\widehat g_\infty,0}\widehat\Phi_\infty=0
\]
weakly. Hence \(\widehat\Phi_\infty\) is smooth and parallel.

It remains to prove nonvanishing. Since \(\overline{\mathcal F}\) is compact
and \(f_\infty>0\), there exists \(c_{\mathcal F}>0\) such that
\[
        f_\infty^{2/(n-1)}\ge c_{\mathcal F}
        \qquad\text{on }\mathcal F.
\]
Using \eqref{eq:conformal-mass-density} and
\eqref{eq:recentered-mass-normalization},
\[
        \int_{\mathcal F}|\widehat\Phi_j|_{\widehat g_\infty}^2\,d\mu_{\widehat g_\infty}
        =
        \int_{\mathcal F}f_\infty^{2/(n-1)}|\Phi_j|^2\dmu
        \ge
        c_{\mathcal F}.
\]
Strong \(L^2_{\loc}\)-convergence therefore gives
\[
        \int_{\mathcal F}|\widehat\Phi_\infty|_{\widehat g_\infty}^2\,d\mu_{\widehat g_\infty}
        \ge
        c_{\mathcal F}>0.
\]
Thus \(\widehat\Phi_\infty\not\equiv0\).
\end{proof}

\begin{proof}[Proof of \Cref{thm:einstein-rigidity}]
By \Cref{prop:recentered-conformal-limit}, the conformal metric
\[
        \widehat g_\infty=f_\infty^{4/(n-1)}g
\]
carries a nonzero parallel spinor, after identifying the \(s=0\) twisting line
\(F=X\times\C\) with its trivial parallel unit section. The standard
curvature identity for parallel spinors then implies that
\(\widehat g_\infty\) is Ricci-flat; see
\cite[Cor. 2.8]{BourguignonHijaziMilhoratMoroianuMoroianu}. Since
\(f_\infty\) is a generalized ground
state and
\[
        \scal_g\equiv -\frac{4n}{n-1}\lambda_0,
\]
the conformal geometric-rigidity criterion
\Cref{thm:conformal-geometric-rigidity} applies. Therefore
\[
        \Ric_g=-\frac{4\lambda_0}{n-1}g.
\]
If \(\lambda_0>0\), the same criterion gives constant sectional curvature
\[
        \sec_g=-\frac{4\lambda_0}{(n-1)^2}.
\]
Since \((X,g)\) is complete and simply connected, it is isometric to real
hyperbolic space with this normalization.
\end{proof}

\subsection{Proofs of the main results}
\label{subsec:proof-main-results}

We first record the additional flatness observation needed in the
zero-spectrum case of \Cref{thm:main-two-form}.

\begin{lemma}
\label{lem:top-degree-flatness}
Let \((M^{2m},g_M)\) be a closed connected oriented Ricci-flat Riemannian
manifold, and let \(\pi:X\to M\) be its universal cover. Suppose that there
exists a closed two-form \(\omega\in\Omega^2(M)\) such that
\[
        \int_M\omega^m\neq 0,
        \qquad
        \pi^*\omega \text{ is exact.}
\]
Then \((M,g_M)\) is flat.
\end{lemma}

\begin{proof}
By the Fischer--Wolf structure theorem for compact Ricci-flat manifolds \cite[Theorem~4.5]{FischerWolf}, there exists a finite Riemannian covering \(p:(\widehat M,\widehat g)\longrightarrow (M,g_M)\) such that
\[
        (\widehat M,\widehat g)\cong (T^q\times N,g_{T^q}\oplus h),
\]
where \(T^q\) is a flat torus and \(N\) is a compact simply connected Ricci-flat manifold.

Let \(\widehat\omega:=p^*\omega\).
Then
\[
        \int_{\widehat M}\widehat\omega^m
        =
        \deg(p)\int_M\omega^m\neq 0.
\]
Thus \([\widehat\omega]^m\neq 0\) in \(H^{2m}(\widehat M;\R)\).

The universal cover of \(T^q\times N\) is \(\R^q\times N\).
Let
\[
        \widehat\pi:\R^q\times N\longrightarrow \widehat M
\]
be the universal covering map. Since \(p\circ\widehat\pi\) identifies with the universal covering map of \(M\), and since \(\pi^*\omega\) is exact, we have \(\widehat\pi^*[\widehat\omega]=\widehat\pi^*p^*[\omega]=0\).
Since \(N\) is simply connected, the K\"unneth formula gives
\[
        H^2(T^q\times N;\R)
        \cong
        H^2(T^q;\R)\oplus H^2(N;\R),
\]
while
\[
        H^2(\R^q\times N;\R)
        \cong
        H^2(N;\R).
\]
Under these identifications, the pullback map to the universal cover is the projection onto the \(H^2(N;\R)\)-summand.

Because \(\widehat\pi^*[\widehat\omega]=0\), the \(H^2(N;\R)\)-component of \([\widehat\omega]\) is zero. 
Therefore there exists a class \(a\in H^2(T^q;\R)\) such that \([\widehat\omega]=\operatorname{pr}_1^*a\), where \(\operatorname{pr}_1:T^q\times N\to T^q\) is the projection.
It follows that \([\widehat\omega]^m=\operatorname{pr}_1^*(a^m)\).
Since \([\widehat\omega]^m\neq 0\), we have \(a^m\neq 0\) in \(H^{2m}(T^q;\R)\).
Therefore \(H^{2m}(T^q;\R)\neq 0\).
This forces \(q\geq 2m\).
But \(\dim(T^q\times N)=2m\), so \(q=2m\) and \(N\) is a point. 
Hence \(\widehat M\) is a flat torus.
Since \(p:\widehat M\to M\) is a finite Riemannian covering and \(\widehat M\) is flat, \(M\) is flat as well.
\end{proof}

\begin{proof}[Proof of \Cref{thm:main-two-form}]
By \Cref{prop:Gamma-s-index-and-spinors}, there exist \(s_j\in(0,1]\), \(s_j\to0\), and nonzero smooth \(L^2\)-spinors
\[
        \psi_j\in C^\infty(X,\Sigma_gX\otimes (X\times\C))
        \cap L^2(X,\Sigma_gX\otimes (X\times\C))
\]
such that \(D_{s_j}\psi_j=0\).
After normalizing each \(\psi_j\) in \(L^2\),
\Cref{thm:scalar-upper-bound} applied to this sequence gives
\[
        \inf_M\operatorname{scal}_{g_M}
        \leq
        -\frac{4(2m)}{2m-1}\lambda_0(X,g).
\]
Assume now that equality holds. 
By \Cref{thm:scalar-upper-bound},
\[
        \operatorname{scal}_g\equiv
        -\frac{4n}{n-1}\lambda_0(X,g).
\]
Thus \Cref{thm:einstein-rigidity} applies and gives
\[
        \operatorname{Ric}_g
        =
        -\frac{4\lambda_0(X,g)}{2m-1}g.
\]
If \(\lambda_0>0\), the same theorem shows that \((X,g)\) is real
hyperbolic with constant sectional curvature
\[
        -\frac{4\lambda_0(X,g)}{(2m-1)^2}.
\]
Finally, if \(\lambda_0=0\) and \(\int_M\omega^m\neq0\), then \(g_M\) is
Ricci-flat and \Cref{lem:top-degree-flatness} shows that it is flat. Hence
\((X,g)\) is isometric to Euclidean space.
\end{proof}

\begin{proof}[Proof of \Cref{thm:main-zero-in-the-spectrum}]
Let \(D\) be the untwisted spin Dirac operator on \((X,g)\). Since
\(0\in\sigma_{L^2}(D)\),
\Cref{prop:zero-in-spectrum-produces-almost-harmonic} provides smooth
compactly supported spinors \(\psi_j\) such that
\[
        \|\psi_j\|_{L^2}=1,
        \qquad
        \|D\psi_j\|_{L^2}\longrightarrow0.
\]
Applying \Cref{thm:scalar-upper-bound} with \(\omega=0\), \(\eta=0\), and
\(s_j=0\) proves the scalar-curvature inequality and its scalar equality
statement. If equality holds, \Cref{thm:einstein-rigidity} gives
\[
        \Ric_g=-\frac{4\lambda_0}{n-1}g.
\]
If \(\lambda_0=0\), this is equivalent to \(g_M\) being Ricci-flat.
Conversely, if \(g_M\) is Ricci-flat, then \(\scal_{g_M}\equiv0\), so
equality holds.

If \(\lambda_0>0\), \Cref{thm:einstein-rigidity} shows that \((X,g)\) is
real hyperbolic with constant sectional curvature
\[
        -\frac{4\lambda_0}{(n-1)^2}.
\]
Conversely, this curvature identity gives
\[
        \scal_g=-\frac{4n}{n-1}\lambda_0,
\]
and hence equality in the scalar-curvature estimate.
\end{proof}

\begin{proof}[Proof of \Cref{cor:compact-enlargeable-rigidity}]
By \Cref{prop:compact-enlargeability-zero-spectrum}, the spin Dirac operator
on \((X,g)\) has zero in its \(L^2\)-spectrum. Hence
\Cref{thm:main-zero-in-the-spectrum}, in the virtually spin form of
\Cref{rem:zero-in-the-spectrum-virtually-spin}, gives the scalar-curvature
inequality.

Suppose equality holds. If \(\lambda_0>0\), the same theorem shows that
\((M,g_M)\) has constant sectional curvature
\(-4\lambda_0/(n-1)^2\). If \(\lambda_0=0\), then \(g_M\) is Ricci-flat.
Applying the Gromov--Lawson rigidity theorem
\cite[Theorem~A]{GromovLawson1980} on a finite spin cover, which is again
compactly enlargeable, shows that the lifted metric is flat. Thus \(g_M\) is
flat as well.
\end{proof}

\bibliographystyle{amsplain}
\bibliography{references}

@article{LongReidVirtuallySpinning,
  author  = {Long, D. D. and Reid, A. W.},
  title   = {Virtually spinning hyperbolic manifolds},
  journal = {Proceedings of the Edinburgh Mathematical Society},
  volume  = {63},
  number  = {2},
  pages   = {305--313},
  year    = {2020},
  doi     = {10.1017/S0013091519000324}
}

@book{doCarmoRiemannianGeometry,
  author = {do Carmo, Manfredo P.},
  title = {Riemannian Geometry},
  series = {Mathematics: Theory \& Applications},
  publisher = {Birkh{\"a}user Boston},
  year = {1992}
}

@misc{Liu26,
  author        = {Liu, Daoqiang},
  title         = {Bottom spectrum, vertical $\widehat{A}$-cowaist and scalar curvature rigidity},
  year          = {2026},
  eprint        = {2605.18269},
  archivePrefix = {arXiv},
  note          = {\newline arXiv:2605.18269},
  doi           = {10.48550/arXiv.2605.18269},
  url           = {https://arxiv.org/abs/2605.18269}
}

@article{Rosenberg1983,
  author  = {Rosenberg, Jonathan},
  title   = {{C}$^*$-algebras, positive scalar curvature, and the {N}ovikov conjecture},
  journal = {Publications Math{\'e}matiques de l'IH{\'E}S},
  volume  = {58},
  year    = {1983},
  pages   = {197--212},
  doi     = {10.1007/BF02953775}
}

@unpublished{SchickCoarseIndex,
  author = {Schick, Thomas},
  title  = {Lectures on coarse index theory},
  year   = {2012},
  note   = {Lecture notes, Cortona},
  url    = {https://wwwuser.gwdguser.de/~tschick/publ/Cortona_12_coarse_index.pdf}
}

@article {Ono1988,
    AUTHOR = {Ono, Kaoru},
     TITLE = {The scalar curvature and the spectrum of the {L}aplacian of
              spin manifolds},
   JOURNAL = {Math. Ann.},
  FJOURNAL = {Mathematische Annalen},
    VOLUME = {281},
      YEAR = {1988},
    NUMBER = {1},
     PAGES = {163--168},
      ISSN = {0025-5831,1432-1807},
   MRCLASS = {58G25 (35P15)},
  MRNUMBER = {944610},
MRREVIEWER = {Hubert\ Gollek},
       DOI = {10.1007/BF01449223},
       URL = {https://doi.org/10.1007/BF01449223},
}

@article{LiWangPositiveSpectrumII,
  author  = {Li, Peter and Wang, Jiaping},
  title   = {Complete manifolds with positive spectrum, {II}},
  journal = {Journal of Differential Geometry},
  volume  = {62},
  number  = {1},
  pages   = {143--162},
  year    = {2002},
  doi     = {10.4310/jdg/1090425532}
}

@article{Brooks1981,
  author  = {Brooks, Robert},
  title   = {The fundamental group and the spectrum of the Laplacian},
  journal = {Commentarii Mathematici Helvetici},
  volume  = {56},
  number  = {4},
  pages   = {581--598},
  year    = {1981},
  doi     = {10.1007/BF02566228}
}

@misc{DiCerboDranishnikovJauhari2025,
  author        = {Di Cerbo, Luca F. and Dranishnikov, Alexander and Jauhari, Ekansh},
  title         = {Curvature, macroscopic dimensions, and symmetric products of surfaces},
  eprint        = {2503.01779},
  archivePrefix = {arXiv},
  year          = {2025},
  note          = {arXiv:2503.01779},
  url           = {https://arxiv.org/abs/2503.01779}
}

@misc{DiCerboDranishnikovJauhari2026,
  author        = {Di Cerbo, Luca F. and Dranishnikov, Alexander and Jauhari, Ekansh},
  title         = {Symplectically aspherical {K}{\"a}hler manifolds, scalar curvature, and the fundamental group},
  year          = {2026},
  eprint        = {2607.05170},
  archivePrefix = {arXiv},
  note          = {arXiv:2607.05170},
  url           = {https://arxiv.org/abs/2607.05170}
}

@article{MathaiLowDegree,
  author  = {Mathai, Varghese},
  title   = {The {Novikov} conjecture for low degree cohomology classes},
  journal = {Geometriae Dedicata},
  volume  = {99},
  pages   = {1--15},
  year    = {2003},
  doi     = {10.1023/A:1024941020306}
}

@article{GuoMathai,
  author  = {Guo, Hao and Mathai, Varghese},
  title   = {Higher localised {$\widehat{A}$}-genera for proper actions and applications},
  journal = {Journal of Functional Analysis},
  volume  = {283},
  number  = {12},
  pages   = {109695},
  year    = {2022},
  doi     = {10.1016/j.jfa.2022.109695},
  url     = {https://doi.org/10.1016/j.jfa.2022.109695}
}

@article{ConnesGromovMoscovici,
  author  = {Connes, Alain and Gromov, Mikha{\"i}l and Moscovici, Henri},
  title   = {Group cohomology with {Lipschitz} control and higher signatures},
  journal = {Geometric and Functional Analysis},
  volume  = {3},
  number  = {1},
  pages   = {1--78},
  year    = {1993},
  doi     = {10.1007/BF01895513}
}

@misc{BeiDiverioTrapani,
  author        = {Bei, Francesco and Claudon, Beno{\^i}t and Diverio, Simone and Trapani, Stefano},
  title         = {Weak {K}\"ahler hyperbolicity is birational},
  year          = {2024},
  eprint        = {2406.01734},
  archivePrefix = {arXiv},
  note          = {arXiv:2406.01734, to appear in \emph{Algebraic Geometry}},
  url           = {https://arxiv.org/abs/2406.01734}
}

@misc{FBeiDiverioTrapani,
  author        = {Bei, Francesco and Diverio, Simone and Trapani, Stefano},
  title         = {Geometric effects of hyperbolic cohomology classes on {K}ähler manifolds (with an appendix by {B}eno{\^i}t {C}laudon)},
  year          = {2025},
  eprint        = {2506.09907},
  archivePrefix = {arXiv},
  note          = {arXiv:2506.09907},
  url           = {https://arxiv.org/abs/2506.09907}
}

@article {HankeSchick,
    AUTHOR = {Hanke, Bernhard and Schick, Thomas},
     TITLE = {The strong {N}ovikov conjecture for low degree cohomology},
   JOURNAL = {Geom. Dedicata},
  FJOURNAL = {Geometriae Dedicata},
    VOLUME = {135},
      YEAR = {2008},
     PAGES = {119--127},
      ISSN = {0046-5755,1572-9168},
   MRCLASS = {19K56},
  MRNUMBER = {2413333},
MRREVIEWER = {Peter\ Haskell},
       DOI = {10.1007/s10711-008-9266-9},
       URL = {https://doi.org/10.1007/s10711-008-9266-9},
}

@article{GromovLawson1980,
  author  = {Gromov, Mikhael and Lawson, Jr., H. Blaine},
  title   = {Spin and scalar curvature in the presence of a fundamental group. {I}},
  journal = {Annals of Mathematics. Second Series},
  volume  = {111},
  number  = {2},
  pages   = {209--230},
  year    = {1980}
}

@book{AschenbrennerFriedlWilton,
  author    = {Aschenbrenner, Matthias and Friedl, Stefan and Wilton, Henry},
  title     = {3-Manifold Groups},
  series    = {EMS Series of Lectures in Mathematics},
  publisher = {European Mathematical Society},
  address   = {Z\"{u}rich},
  year      = {2015},
  doi       = {10.4171/154}
}

@article{Marques2012,
  author  = {Marques, Fernando Cod\'{a}},
  title   = {Deforming three-manifolds with positive scalar curvature},
  journal = {Annals of Mathematics. Second Series},
  volume  = {176},
  number  = {2},
  pages   = {815--863},
  year    = {2012},
  doi     = {10.4007/annals.2012.176.2.3}
}

@article{HankeSchickEnlargeability,
  author  = {Hanke, Bernhard and Schick, Thomas},
  title   = {Enlargeability and index theory: infinite covers},
  journal = {$K$-Theory},
  volume  = {38},
  number  = {1},
  pages   = {23--33},
  year    = {2007},
  doi     = {10.1007/s10977-007-9004-3}
}

@article{Wang2008HyperbolicSpace,
  author  = {Wang, Xiaodong},
  title   = {Harmonic Functions, Entropy, and a Characterization of the Hyperbolic Space},
  journal = {The Journal of Geometric Analysis},
  volume  = {18},
  number  = {1},
  pages   = {272--284},
  year    = {2008},
  doi     = {10.1007/s12220-007-9001-z}
}

@misc{WangZhu2026,
  author        = {Wang, Jinmin and Zhu, Bo},
  title         = {Scalar curvature, sharp bottom spectrum and geometric rigidity},
  year          = {2026},
  eprint        = {2606.11957},
  archivePrefix = {arXiv},
  note          = {arXiv:2606.11957},
  url           = {https://arxiv.org/abs/2606.11957}
}

@article{LedrappierWang,
    AUTHOR = {Ledrappier, Fran\c{c}ois and Wang, Xiaodong},
     TITLE = {An integral formula for the volume entropy with applications
              to rigidity},
   JOURNAL = {J. Differential Geom.},
  FJOURNAL = {Journal of Differential Geometry},
    VOLUME = {85},
      YEAR = {2010},
    NUMBER = {3},
     PAGES = {461--477},
      ISSN = {0022-040X,1945-743X},
   MRCLASS = {53C24 (53C20)},
  MRNUMBER = {2739810},
MRREVIEWER = {Nelia\ Charalambous},
       URL = {http://projecteuclid.org/euclid.jdg/1292940691},
}

@book{Bal06,
    AUTHOR = {Ballmann, Werner},
     TITLE = {Lectures on {K}\"{a}hler manifolds},
    SERIES = {ESI Lectures in Mathematics and Physics},
 PUBLISHER = {European Mathematical Society (EMS), Z\"{u}rich},
      YEAR = {2006},
     PAGES = {x+172},
      ISBN = {978-3-03719-025-8; 3-03719-025-6},
   MRCLASS = {32Q15 (53C55)},
  MRNUMBER = {2243012},
MRREVIEWER = {Joel\ Fine},
       DOI = {10.4171/025},
       URL = {https://doi.org/10.4171/025},
}

@article{WangProperCocompactIndex,
  author  = {Wang, Hang},
  title   = {{$L^2$-index formula for proper cocompact group actions}},
  journal = {J. Noncommut. Geom.},
  volume  = {8},
  number  = {2},
  pages   = {393--432},
  year    = {2014},
  doi     = {10.4171/JNCG/160}
}

@article{Cheng1975,
  author  = {Cheng, Shiu Yuen},
  title   = {Eigenvalue comparison theorems and its geometric applications},
  journal = {Mathematische Zeitschrift},
  volume  = {143},
  number  = {3},
  pages   = {289--297},
  year    = {1975},
  doi     = {10.1007/BF01214381},
  mrnumber = {0378001}
}

@article{MunteanuWang2024,
  author  = {Munteanu, Ovidiu and Wang, Jiaping},
  title   = {Bottom spectrum of three-dimensional manifolds with scalar curvature lower bound},
  journal = {Journal of Functional Analysis},
  volume  = {287},
  number  = {2},
  pages   = {Paper No. 110457},
  year    = {2024},
  doi     = {10.1016/j.jfa.2024.110457},
  mrnumber = {4736650}
}

@misc{MunteanuWang2026,
  author        = {Munteanu, Ovidiu and Wang, Jiaping},
  title         = {Bottom spectrum and parabolicity of 3-manifolds with scalar curvature lower bound},
  year          = {2026},
  eprint        = {2607.06508},
  archivePrefix = {arXiv},
  note          = {arXiv:2607.06508},
  url           = {https://arxiv.org/abs/2607.06508}
}

@article{Gromov1991,
  author  = {Gromov, Mikhael},
  title   = {{K}{\"a}hler hyperbolicity and {$L^2$}-{H}odge theory},
  journal = {Journal of Differential Geometry},
  volume  = {33},
  number  = {1},
  year    = {1991},
  pages   = {263--292},
  doi     = {10.4310/jdg/1214446039}
}

@incollection{GromovFourLectures,
    AUTHOR = {Gromov, Misha},
     TITLE = {Four lectures on scalar curvature},
 BOOKTITLE = {Perspectives in scalar curvature. {V}ol. 1},
     PAGES = {1--514},
 PUBLISHER = {World Sci. Publ., Hackensack, NJ},
      YEAR = {[2023] \copyright 2023},
      ISBN = {978-981-124-998-3; 978-981-124-935-8; 978-981-124-936-5},
   MRCLASS = {53C23 (53-02 53C21)},
  MRNUMBER = {4577903},
MRREVIEWER = {Mikhail\ G.\ Katz},
       DOI = {10.1142/9789811273223_0001},
       URL = {https://doi.org/10.1142/9789811273223_0001},
}

@article{ChernoffEssentialSelfAdjointness,
  author  = {Chernoff, Paul R.},
  title   = {Essential self-adjointness of powers of generators of hyperbolic equations},
  journal = {Journal of Functional Analysis},
  volume  = {12},
  number  = {4},
  year    = {1973},
  pages   = {401--414},
  doi     = {10.1016/0022-1236(73)90003-7}
}

@book{ReedSimonI,
  author    = {Reed, Michael and Simon, Barry},
  title     = {Methods of Modern Mathematical Physics. {I}: Functional Analysis},
  edition   = {Revised and enlarged},
  publisher = {Academic Press},
  address   = {New York},
  year      = {1980},
  isbn      = {978-0-12-585050-6}
}

@book{ChavelEigenvalues,
  author    = {Chavel, Isaac},
  title     = {Eigenvalues in {R}iemannian Geometry},
  series    = {Pure and Applied Mathematics},
  volume    = {115},
  publisher = {Academic Press},
  address   = {Orlando, FL},
  year      = {1984},
  note      = {Including a chapter by Burton Randol; with an appendix by Jozef Dodziuk}
}

@article{StrichartzLaplacian,
  author  = {Strichartz, Robert S.},
  title   = {Analysis of the {L}aplacian on the complete {R}iemannian manifold},
  journal = {Journal of Functional Analysis},
  volume  = {52},
  number  = {1},
  year    = {1983},
  pages   = {48--79},
  doi     = {10.1016/0022-1236(83)90090-3}
}

@book{LawsonMichelsohn,
  author    = {Lawson, H. Blaine and Michelsohn, Marie-Louise},
  title     = {Spin Geometry},
  series    = {Princeton Mathematical Series},
  volume    = {38},
  publisher = {Princeton University Press},
  address   = {Princeton, NJ},
  year      = {1989},
  isbn      = {978-0-691-08542-5}
}

@book{BourguignonHijaziMilhoratMoroianuMoroianu,
  author    = {Bourguignon, Jean-Pierre and Hijazi, Oussama and Milhorat, Jean-Louis and Moroianu, Andrei and Moroianu, Sergiu},
  title     = {A Spinorial Approach to {R}iemannian and Conformal Geometry},
  series    = {EMS Monographs in Mathematics},
  publisher = {European Mathematical Society},
  address   = {Z\"{u}rich},
  year      = {2015},
  doi       = {10.4171/136},
  isbn      = {978-3-03719-136-1}
}

@article{LenzStollmannVeselic2009,
  author  = {Lenz, Daniel and Stollmann, Peter and Veseli{\'c}, Ivan},
  title   = {The {Allegretto}--{Piepenbrink} theorem for strongly local {Dirichlet} forms},
  journal = {Documenta Mathematica},
  volume  = {14},
  pages   = {167--189},
  year    = {2009},
  doi     = {10.4171/DM/269}
}

@article{Hijazi1986,
  author  = {Hijazi, Oussama},
  title   = {A conformal lower bound for the smallest eigenvalue of the {D}irac operator and {K}illing spinors},
  journal = {Communications in Mathematical Physics},
  volume  = {104},
  number  = {1},
  year    = {1986},
  pages   = {151--162},
  doi     = {10.1007/BF01210797}
}

@article{FriedrichKim2001,
  author  = {Friedrich, Thomas and Kim, Eui Chul},
  title   = {Some remarks on the {H}ijazi inequality and generalizations of the {K}illing equation for spinors},
  journal = {Journal of Geometry and Physics},
  volume  = {37},
  number  = {1--2},
  year    = {2001},
  pages   = {1--14},
  doi     = {10.1016/S0393-0440(99)00049-2}
}

@article{Friedrich1989,
  author  = {Friedrich, Thomas},
  title   = {On the conformal relation between twistors and {K}illing spinors},
  journal = {Rendiconti del Circolo Matematico di Palermo. Serie II. Supplemento},
  number  = {22},
  year    = {1989},
  pages   = {59--75}
}

@book{GilbargTrudinger,
  author    = {Gilbarg, David and Trudinger, Neil S.},
  title     = {Elliptic Partial Differential Equations of Second Order},
  edition   = {2},
  series    = {Classics in Mathematics},
  publisher = {Springer},
  address   = {Berlin},
  year      = {2001},
  note      = {Reprint of the 1983 second edition}
}

@article{CalderbankGauduchonHerzlich,
  author  = {Calderbank, David M. J. and Gauduchon, Paul and Herzlich, Marc},
  title   = {Refined {K}ato inequalities and conformal weights in {R}iemannian geometry},
  journal = {Journal of Functional Analysis},
  volume  = {173},
  number  = {1},
  year    = {2000},
  pages   = {214--255},
  doi     = {10.1006/jfan.2000.3563}
}

@article{Davaux2003,
    AUTHOR = {Davaux, H\'{e}l\`{e}ne},
     TITLE = {An optimal inequality between scalar curvature and spectrum of
              the {L}aplacian},
   JOURNAL = {Math. Ann.},
  FJOURNAL = {Mathematische Annalen},
    VOLUME = {327},
      YEAR = {2003},
    NUMBER = {2},
     PAGES = {271--292},
      ISSN = {0025-5831,1432-1807},
   MRCLASS = {58J50 (46L10 53C21 58J20 58J30)},
  MRNUMBER = {2015070},
MRREVIEWER = {Thomas\ Schick},
       DOI = {10.1007/s00208-003-0451-8},
       URL = {https://doi.org/10.1007/s00208-003-0451-8},
}

@article{Brinkmann1925,
  author  = {Brinkmann, H. W.},
  title   = {Einstein spaces which are mapped conformally on each other},
  journal = {Mathematische Annalen},
  volume  = {94},
  number  = {1},
  year    = {1925},
  pages   = {119--145},
  doi     = {10.1007/BF01208647}
}

@incollection{KuhnelRademacher2008,
  author    = {K{\"u}hnel, Wolfgang and Rademacher, Hans-Bert},
  title     = {Conformal transformations of pseudo-{R}iemannian manifolds},
  booktitle = {Recent Developments in Pseudo-Riemannian Geometry},
  editor    = {Alekseevsky, Dmitri V. and Baum, Helga},
  series    = {ESI Lectures in Mathematics and Physics},
  pages     = {261--298},
  publisher = {European Mathematical Society},
  address   = {Z{\"u}rich},
  year      = {2008},
  doi       = {10.4171/051-1/8}
}

@misc{WangZhu2024,
  author        = {Wang, Jinmin and Zhu, Bo},
  title         = {Sharp bottom spectrum and scalar curvature rigidity},
  year          = {2024},
  eprint        = {2408.08245},
  archivePrefix = {arXiv},
  note          = {arXiv:2408.08245},
  url           = {https://arxiv.org/abs/2408.08245}
}

@article{FischerColbrieSchoen1980,
  author  = {Fischer-Colbrie, Doris and Schoen, Richard},
  title   = {The structure of complete stable minimal surfaces in 3-manifolds of non-negative scalar curvature},
  journal = {Communications on Pure and Applied Mathematics},
  volume  = {33},
  number  = {2},
  year    = {1980},
  pages   = {199--211},
  doi     = {10.1002/cpa.3160330206}
}

@book{Besse,
  author    = {Besse, Arthur L.},
  title     = {Einstein Manifolds},
  series    = {Ergebnisse der Mathematik und ihrer Grenzgebiete, 3. Folge},
  volume    = {10},
  publisher = {Springer-Verlag},
  address   = {Berlin},
  year      = {1987},
  isbn      = {978-3-540-15279-8}
}

@article{FischerWolf,
  author  = {Fischer, Arthur E. and Wolf, Joseph A.},
  title   = {The structure of compact {R}icci-flat {R}iemannian manifolds},
  journal = {Journal of Differential Geometry},
  volume  = {10},
  number  = {2},
  year    = {1975},
  pages   = {277--288}
}

@article {JostZuo,
    AUTHOR = {Jost, J\"urgen and Zuo, Kang},
     TITLE = {Vanishing theorems for {$L^2$}-cohomology on infinite
              coverings of compact {K}\"ahler manifolds and applications in
              algebraic geometry},
   JOURNAL = {Comm. Anal. Geom.},
  FJOURNAL = {Communications in Analysis and Geometry},
    VOLUME = {8},
      YEAR = {2000},
    NUMBER = {1},
     PAGES = {1--30},
      ISSN = {1019-8385,1944-9992},
   MRCLASS = {32L20 (32Q30 58J20)},
  MRNUMBER = {1730897},
MRREVIEWER = {Philippe\ P.\ Eyssidieux},
       DOI = {10.4310/CAG.2000.v8.n1.a1},
       URL = {https://doi.org/10.4310/CAG.2000.v8.n1.a1},
}

@article {CaoXavier,
    AUTHOR = {Cao, Jianguo and Xavier, Frederico},
     TITLE = {K\"ahler parabolicity and the {E}uler number of compact
              manifolds of non-positive sectional curvature},
   JOURNAL = {Math. Ann.},
  FJOURNAL = {Mathematische Annalen},
    VOLUME = {319},
      YEAR = {2001},
    NUMBER = {3},
     PAGES = {483--491},
      ISSN = {0025-5831,1432-1807},
   MRCLASS = {53C21 (32Q05 32Q15 53C55)},
  MRNUMBER = {1819879},
MRREVIEWER = {Mar\'ia\ J.\ Druetta},
       DOI = {10.1007/PL00004444},
       URL = {https://doi.org/10.1007/PL00004444},
}

@article {BDET,
    AUTHOR = {Bei, Francesco and Diverio, Simone and Eyssidieux, Philippe
              and Trapani, Stefano},
     TITLE = {Weakly {K}\"ahler hyperbolic manifolds and the
              {G}reen-{G}riffiths-{L}ang conjecture},
   JOURNAL = {J. Reine Angew. Math.},
  FJOURNAL = {Journal f\"ur die Reine und Angewandte Mathematik. [Crelle's
              Journal]},
    VOLUME = {807},
      YEAR = {2024},
     PAGES = {257--297},
      ISSN = {0075-4102,1435-5345},
   MRCLASS = {53C55 (32Q45)},
  MRNUMBER = {4698497},
MRREVIEWER = {Riccardo\ Piovani},
       DOI = {10.1515/crelle-2023-0094},
       URL = {https://doi.org/10.1515/crelle-2023-0094},
}

@article{AtiyahL2Index,
  author  = {Atiyah, Michael F.},
  title   = {Elliptic operators, discrete groups and von {N}eumann algebras},
  journal = {Ast\'{e}risque},
  number  = {32--33},
  year    = {1976},
  pages   = {43--72}
}

\end{document}